\documentclass[11pt]{amsart}

\usepackage{fullpage} 
\usepackage{mathrsfs} 
\usepackage{amsmath, amssymb, amsthm}
\usepackage{tikz-cd} 
\usepackage{hyperref}
\hypersetup{
	colorlinks = true,
	linkbordercolor = {red},
	citebordercolor = {blue},
}

\usepackage[alphabetic]{amsrefs}
\usepackage[title]{appendix}
\usepackage{geometry}
\allowdisplaybreaks

\numberwithin{equation}{section}

\theoremstyle{definition}
\newtheorem{thm}[equation]{Theorem} 
\theoremstyle{definition}

\def\fnum{equation}
\newtheorem{Thm}[\fnum]{Theorem}
\newtheorem{Cor}[\fnum]{Corollary}
\newtheorem{Key Lemma}[\fnum]{Key Lemma}

\newtheorem{Lem}[\fnum]{Lemma}

\newtheorem{Rem}[\fnum]{Remark}
\newtheorem{Pro}[\fnum]{Proposition}

\newcommand{\RR}{\mathbb R}

\newcommand{\ep}{\varepsilon}

\newcommand{\na}{\nabla}
\newcommand{\la}{\langle}
\newcommand{\ra}{\rangle}

\newcommand{\M}{\mathcal M}

\title{On the rate of convergence of cylindrical singularity in mean curvature flow}
\author{Yiqi Huang and Xinrui Zhao}
\address{MIT, Dept. of Math., 77 Massachusetts Avenue, Cambridge, MA 02139-4307}
\address{Department of Mathematics
Yale University, 219 Prospect St, Floors 7–9, New Haven, CT 06511}
\email{yiqih777@mit.edu and xinrui.zhao@yale.edu}

\begin{document}
	\begin{abstract}
We prove that if a rescaled mean curvature flow is a global graph over the round cylinder with small gradient and converges super-exponentially fast, then it must coincide with the cylinder itself. We also show that the result is sharp with counter-examples of local graphs at arbitrarily super-exponential convergence rate with the domain expanding arbitrarily fast. 

The first part provides the first unique continuation result in the cylindrical setting, the generic singularity model in mean curvature flow. In sharp contrast, in the second part we construct infinite-dimensional families of Tikhonov-type examples for nonlinear equations, including the rescaled mean curvature flow, showing that unique continuation fails for local graphical solutions. These examples demonstrate the essential role of global graphical assumptions in rigidity and highlight new phenomena absent in the compact case. We also construct non-product mean curvature flows that develop singular sets as prescribed lower dimensional Euclidean space at arbitrary super-exponential rates. Our construction works in great generality for a large class of non-linear equations.
\end{abstract}

\maketitle

\tableofcontents

\vspace{-1cm}

\section{Introduction}

In this paper, we study the convergence rate and unique continuation problem for mean curvature flow, which is a geometric evolution equation where there is a family of immersions $\varphi \colon M^n \times I \to \mathbb{R}^{n+k}$ evolving by the mean curvature 
\[\frac{\partial \tilde{\varphi}}{\partial t} = \overrightarrow{H},
\]
with $\overrightarrow{H}$ denoting the mean curvature vector of the submanifold $M_t = \varphi(M, t)$.

As singularities inevitably occur for mean curvature flows starting from closed hypersurfaces, understanding their formation and structure is a crucial and central question. A standard approach is to study singularities via blowup analysis. If $(0,0) \in \mathbb{R}^{n+k} \times \mathbb{R}$ is a singularity of the flow, one can rescale the flow by $\lambda_i \to \infty$ centered at $(0,0)$ to obtain $M_t^i := \lambda_i M_{\lambda_i^{-2}t}$. By Huisken's monotonicity formula \cite{H90} and Brakke's compactness theorem \cite{B78}, we know that there is a subsequence that converges to a self-similar limit flow - a \emph{tangent flow}. Moreover, if the limit flow is smooth, it is generated by rescaling a hypersurface called \emph{self-shrinker}, satisfying $\mathbf{H} = -\frac{x^\perp}{2}$. Equivalently, \emph{self-shrinkers} are the static solutions for the following \emph{rescaled mean curvature flow} equation
\begin{align}\label{e:recaled mcf}
        \frac{\partial \tilde{\varphi}}{\partial t} = \overrightarrow{H}+\frac{\tilde{\varphi}^\perp}{2},
    \end{align}

A priori, the tangent flow may depend on the choice of rescaling sequence. The uniqueness of the tangent flows, as one of the central problems in the analysis of the singularities, has been resolved in several important cases, \cites{S14, CM15,CS, LSS, CM23,S23, Z20,LZ,CSSZ}. Closely related to uniqueness problem is the question of 
the convergence rate of the rescaled mean curvature flows toward the tangent flow, particularly the \textbf{unique continuation problem}. That is, if a mean curvature flow approaches a tangent flow sufficiently fast, does this 
force the flow to be identical to the tangent flow?

\subsection{Unique continuation for global graphs}

The unique continuation problem has already been extensively studied for elliptic and parabolic equations. See \cites{Car39, AKS62,GL86,GL87,P96,ESS03} and the references therein. In the mean curvature flow setting, the spherical case was studied by Sesum \cite{S08} and Strehlke \cite{S20}. In \cite{S20}, he proved that if the rescaled mean curvature flow converges to a round sphere super-exponentially, then it must a round sphere itself. In \cite{MS23}, Martin-Hagemayer and Sesum generalized this unique continuation property to any compact shrinkers. For more recent results in this direction, see for example \cites{DH25,DH252,SWX25}. See also the related question in Ricci flows in \cites{SW15, Kot22,CL25}.

In this paper, we focus on the unique continuation problem for singularities modeled on cylinders. Cylinders are of particular importance, as they are the only \emph{generic} shrinkers \cites{CM12,CIMW13,ALW14,CM151,CM16,B16,BW17,CM19,BC19,CM21,BC21,CM23,CCMS24}. 
At the same time, they present essential new challenges. Unlike compact shrinkers, the noncompactness nature of cylinders creates substantial analytic difficulties. Compared with other noncompact shrinkers, such as conical singularities, the cylinder is even more degenerate: it gives rise to non-isolated singularities and requires delicate control of the flow in annular regions around the singular set. See for example \cites{W97,W00,W14,GK15,BW17,GKS18,KKM,CS,Z21,CHHW22,LZ}. These features make the cylindrical case qualitatively harder and correspondingly more intriguing. See a related unique continuation conjecture in \cite{SWX25}*{Conjecture 1.6}.

Our first main result confirms the unique continuation property for global graphs in the cylindrical setting.

\begin{Thm}\label{t:main}
     Let $\tilde{\varphi}:M\times[0,\infty)\to \RR^{n+1}$ be a solution to the rescaled mean curvature flow equation \eqref{e:recaled mcf} such that $M_t \equiv \tilde{\varphi}(M,t)$ can be expressed as a normal graph of a function $u(x,t)$ over the round cylinder $\mathcal{C}_k=\mathbb{S}^{n-k}_{\sqrt{2(n-k)}}\times \RR^{k}$. Then there exists some $\ep_0(n)$ such that the following holds: if $\|u\|_{C^1(\mathcal{C}_k \times [0, \infty))}\leq \ep_0$ and we have
    \begin{equation*}
    \lim_{t\to \infty} e^{\alpha t}\int_{\mathcal{C}_k} u^2(x,t) e^{-\frac{|x|^2}{4}}d\mu =0 \text{ for any } \alpha\in \mathbb{N}
\end{equation*}   
where $d\mu$ is the standard measure on $\mathcal{C}_k$, then we have $u \equiv 0$.
\end{Thm}

This implies that if $u \not\equiv 0$, then there exists some \emph{exponential upper bound} for the convergence rate of $u$.

Note that in Theorem~\ref{t:main} we only assume smallness the $C^1$-norm. If, in addition, we assume the $C^1(\mathcal{C}_k)$-norm of $u$ converges to $0$, then we can obtain a \emph{uniform exponential upper bound} for the convergence rate in each small region $\Omega \subset \mathcal{C}_k$. This in particular implies the unique continuation property even with respect to \emph{local} convergence.
\begin{thm}\label{t:local}
     Let $\tilde{\varphi}:M\times[0,\infty)\to \RR^{n+1}$ be a solution to the rescaled mean curvature flow equation* \eqref{e:recaled mcf} and $M_t$ can be expressed as a normal graph of a function $u(p,t)$ over $\mathcal{C}_k=\mathbb{S}^{n-k}_{\sqrt{2(n-k)}}\times \RR^{k}$ such that $\lim\limits_{t\to \infty}\|u(\cdot,t)\|_{C^1(\mathcal{C}_k)}= 0$. If $u\not\equiv 0$, then there exists an $\alpha_0$ such that for every open set $\Omega\subset \mathcal{C}_k$
\begin{align*}
    \int_{\Omega}u^2(x,t) \,e^{-\frac{|x|^2}{4}}d\mu \geq C(\Omega,u)e^{-\alpha_0t}.
\end{align*}
where $d\mu$ is the standard measure on $\mathcal{C}_k$. 
\end{thm}
\begin{Rem}
One can in fact deduce the following unique continuation result: Given $\|u(\cdot,t)\|_{C^1(\mathcal{C}_k)} \to 0$, if there exists an open subset $\Omega\subset \mathcal{C}_k$ such that for any $\alpha>0$ there is a constant $C_{\alpha} < \infty$ satisfying
     \begin{align*}
    \int_{\Omega}u^2(x,t)\,e^{-\frac{|x|^2}{4}} d\mu \leq C_\alpha e^{-\alpha t},
\end{align*}
then we have $u\equiv 0$. 
\end{Rem}

Our approach is different from the frequency function techniques used in \cites{MS23,BHL24,DH25} or the Carleman-type estimate methods of \cites{W14,W16,DH252}. In this paper, we employ a spectral decomposition based on eigenfunctions and use the ODE Lemma due to Filippas-Kohn \cite{FK92} and Merle-Zaag \cite{MZ98} as the main analytical tool. This framework exploits the spectral gap of the linearized operator $L$ and product structure on cylinder, both of which reflect its underlying geometric features. For more references, see for  \cites{ADS19,ADS20,SX22,CHH22,CM22,CCS23, CCMS24,CHH24, ADS25,DZ22,SWX25,CDDHS25}. Moreover, our method applies to a broad class of nonlinear equations.

\subsection{Failure of unique continuation for local graphs}

Our main rigidity theorems, Theorem~\ref{t:local}, require only local super-exponential convergence in order to conclude rigidity, provided that the flow is assumed to be a \emph{global graph} over the entire cylinder.  
A natural question is whether this global graphical assumption can be weakened to \emph{local graphs}.  
In this section, as our second main result, we show that this is not the case.  

We construct explicit examples of rescaled mean curvature flows that are defined only as local graphs, yet converge to cylinders at arbitrarily fast super-exponential rate with their domains of definition expanding arbitrarily fast. 
The graphical region can be chosen arbitrarily large, and expanding arbitrarily rapid, while the flow remains nontrivial.  
This demonstrates that unique continuation at cylindrical singularities is \emph{not} a local property.

Unlike the harmonic case, where strong unique continuation holds under very mild assumptions, the situation is dramatically different for parabolic equations.  
Even for the linear heat equation, striking counterexamples arise once one allows uncontrolled growth at spatial infinity.  
This classical phenomenon, discovered by Tikhonov~\cite{T35}, shows that there exist nontrivial smooth solutions $ u $ of the heat equation on a time half-line that vanish to infinite order at the initial time:
\[
u(x,0) \equiv 0, \qquad \partial_t^k u(x,0) \equiv 0 \quad \text{for all } k \ge 0,
\]
yet \( u \not\equiv 0 \) for \( t \neq 0 \).

One explicit example is obtained by taking a time-flat germ
\[
\varphi(t) = e^{-1/t^2}, \qquad \varphi(0) = 0,
\]
and defining
\begin{equation}\label{e:Tik example}
    u(x,t) = \sum_{k=0}^{\infty} \frac{\varphi^{(k)}(t)\, x^{2k}}{(2k)!}.
\end{equation}
Here all time derivatives vanish at \( t = 0 \), so that the vanishing order at the origin is infinite, yet the resulting solution is nontrivial for \( t>0 \).

A natural and more challenging question is to study the unique continuation problem for
nonlinear parabolic equations. The construction by Tikhonov provides the prototype for our nonlinear generalization in Theorem~\ref{t:Tikhonov}.
In the following theorem, we establish such a result for the \emph{local solutions} to a broad class of nonlinear backward parabolic equations
whose linearized operator is the Laplacian and whose quadratic term is analytic. The construction also works for forward equation.

\begin{Thm}\label{t:Tikhonov}
Let $\mathcal{Q}$ be a real-analytic function defined in a neighborhood of \( (0,0,0) \in \mathbb{R} \times \mathbb{R}^{n} \times \mathbb{R}^{n\times n} \) satisfying  
\[
\mathcal{Q}(0,0,0) = 0, \qquad \nabla \mathcal{Q}(0,0,0) = 0.
\]
Then for any parameters \( \vec{\alpha}=(\alpha_1,\alpha_2,\alpha_3) \in \mathbb{R}^3, \beta \in \mathbb{R} \), $N\geq 1$, there exist a smooth one-parameter family of functions  
\[
\{ u_s : U_N \subset \mathbb{R}^{n} \times [0,1) \to \mathbb{R} \}_{s \in [-1,1]},
\]
such that each \( u_s \) is nontrivial, \( u_s \not\equiv 0 \), and satisfies the equation
\begin{align}\label{e:intro_vars}
    \begin{cases}
        (\partial_t + \Delta) u_s
        = s \cdot t^{\beta} \cdot
       \mathcal{Q}\!\left( t^{\alpha_1} u_s,\; t^{\alpha_2} \nabla u_s,\; t^{\alpha_3} D^2 u_s \right), \\[4pt]
        u_s(x,0) = 0,
    \end{cases}
\end{align}
on the region 
\begin{align}\label{e:domain}
     U_N  \supseteq \bigg\{ (x,t) \in \mathbb{R}^{n} \times (0,1)\;\bigg|\;
      |x| \le e^{\circ N}\!\left( \tfrac{1}{t^2} \right) \bigg\} ,
\end{align}
with the decay rate estimate
\begin{align}\label{e:superexp_decay}
    \limsup_{t\to 0}
    |u_s(x,t)| \cdot e^{\circ N}\!\left( \tfrac{1}{t^2} \right) =0
\end{align}
where the iterated exponential \( e^{\circ N} \) is defined recursively by
\begin{equation}\label{e:iterated exp}
\begin{split}
    e^{\circ 0}(x) &= x, \\
    e^{\circ N}(x) &= e^{\, e^{\circ (N-1)}(x)}.
\end{split}
\end{equation}
\end{Thm}

\begin{Rem}
    By taking $N$ large, we can construct solutions that decay at arbitrarily fast rates on domains expanding at arbitrarily fast rates. Moreover, the same argument applies if we replace $t$ by any smooth function of $t$. Note that to construct such solution defined globally is impossible due to the uniqueness result in \cite{DS23}. Hence our result is sharp.
\end{Rem}

In particular, for the mean curvature flow over the plane,
the corresponding quadratic term
\[
\mathcal{Q}(u, \nabla u, D^2 u) = -\frac{D^2 u(\nabla u, \nabla u)}{1 + |\nabla u|^2}
\]
satisfies the structural assumptions of Theorem~\ref{t:Tikhonov}.
It is worth noting that the solution domain corresponds to a space–time region
that can expand arbitrarily fast in time under the rescaled mean curvature flow.
This observation further indicates that unique continuation
is not an inherently local property in this setting.

Consequently, for the rescaled mean curvature flow, there exist incomplete graphical solutions that converge to zero at arbitrarily fast super-exponential rates with the domain expanding arbitrarily fast.
More precisely, we have the following.

\begin{Cor}\label{c:RMCF}
For any $N \ge 1$, there exists a nontrivial solution $M_t$ to the 
rescaled mean curvature flow \eqref{e:recaled mcf} such that each $M_t$ can be expressed as a normal graph of a function $u(x,t)$ over $\mathcal{C}_k=\mathbb{S}^{n-k}_{\sqrt{2(n-k)}}\times \RR^{k}$ with $k\ge 1$  defined on
\begin{align*}
    \big\{ (x, \tau) \in \mathcal{C}_k \times \mathbb{R}
    \;\big|\;
    |x| \le e^{\circ N}(\tau),\;
    0 \le \tau < +\infty \big\},
\end{align*}
where the iterated exponential $e^{\circ N}$ is defined as \eqref{e:iterated exp}. Moreover, the solution satisfies the super-exponential decay estimate
\begin{align}\label{e:rate}
    \limsup_{\tau \to \infty} |u(x, \tau)| \cdot {e^{\circ N}(\tau)} = 0.
\end{align}
\end{Cor}

Again, such local graphical constructions cannot be extended to
global solutions due to the uniqueness result as in \cite{DS23}.

By considering local graphs on the shrinking cylindrical flow $\RR \times \mathbb{S}^1_{\sqrt{-2t}}$, we can construct infinite-dimensional families of mean curvature flow, though incomplete, developing singularity at time $0$ whose singular set is $\RR$. See also \cite{W16} for an incomplete solution developing singularity whose singular set is a ray $(0,\infty)$.

\begin{Cor}\label{c:prod}
    There exist infinite-dimensional families of non-product mean curvature flows developing singularity at time $0$ with singular set $\RR$. 
\end{Cor}
\begin{Rem}
 For any $N\geq 1$, we can construct a corresponding solution to mean curvature flow, which is a graph over the region on shrinking cylinders $ \{
(x,t)\,\big|\,\, |x|\leq e^{\circ N}(\frac{1}{t^2}),\,\,-1\leq t< 0\} \subset \mathcal{C}_1 \times \RR_{\geq -1}$ and the graph function has super-exponential convergence rate \eqref{e:rate}. By considering the product with more $\RR$ factors, it would provide more examples of mean curvature flows developing singular set $\RR^m$, where the mean curvature flow is not of the form $\RR^m\times \Sigma_t$.
\end{Rem}

\medskip

The difficulty in constructing these solutions arises not only from the nonlinear nature of the equation,
but also from the ill-posedness of the backward parabolic problem.
To overcome this, we introduce a new method combining
\emph{power series expansions} with \emph{combinatorial analysis} to construct the solution explicitly.
The essential step is to prove convergence of the resulting formal series.

Since $\mathcal{Q}$ is analytic, it admits a convergent power series expansion in the variables 
$u$, $\nabla u$, and $D^2 u$. Substituting this expansion into \eqref{e:intro_vars}
yields an iterative system of equations for the Taylor coefficients of $u$.
By writing out the recurrence relations among these coefficients,
the problem reduces to a \emph{counting argument} stemming from the quadratic nature of $\mathcal{Q}$:
we show that the combinatorial growth of coefficients is dominated
by the super-exponential decay contributed by the base term $u_0$,
which we take to be a Tikhonov-type solution of the heat equation.
This combinatorial control guarantees convergence of the series and hence
the existence of the desired smooth solution.

\medskip

\textbf{Acknowledgments. }
The authors are very grateful to Tobias Colding for his invaluable support and encouragement throughout this project. The authors are further indebted to Wenshuai Jiang, Natasa Sesum and Lu Wang for their kind encouragement and interest in this work. The authors would like to thank Zhihan Wang for his generous help and insightful suggestions regarding the construction of counterexamples modeled on Tikhonov’s example and Ao Sun, Jingxuan Zhang, Jingze Zhu for several helpful discussions. The authors also thank Otis Chodosh, Tang-Kai Lee, Ao Sun, Zhihan Wang and Jingxuan Zhang for several useful comments on an earlier version of this paper. During the project, Yiqi Huang was supported by the Croucher Scholarship and NSF Grant DMS 2405393, and Xinrui Zhao was supported by NSF Grant DMS 1812142, NSF Grant DMS 1811267 and NSF Grant DMS 2104349.

\section{Rescaled mean curvature flow over cylinders}

In this section, we will review and prove some estimates for the quadratic term $\mathcal{Q}$ for rescaled mean curvature flows equation of global graphs $u$ over the cylinder $\mathcal{C}_k$ as in Theorem \ref{t:main}. Suppose $|u|_{C^1(\mathcal{C}_k \times [0, \infty)} \le \ep_0$ for some small $\ep_0>0$ to be determined later. 

From \cite{SX22}*{Proposition A.1}, we can write the equation as
\begin{equation*}
    \partial_tu=Lu+\mathcal{Q}(u,\na u,\na^2u)
\end{equation*}
where $L$ is the linearized operator
\begin{align*}
    Lf=\Delta f-\frac{x}{2}\cdot \na f+f,
\end{align*}
and $\mathcal{Q}$ is the quadratic term satisfying
\begin{align}\label{e:Q}
    \mathcal{Q}(u,\na u,\na^2u)=-C(u^2+4u\Delta_\theta u+2|\na_{\theta}u|^2)+\tilde{\mathcal{Q}}(u,\na u,\na^2 u).
\end{align}
Here $\na_\theta$ and $\Delta_\theta$ mean that we are taking derivatives with respect to the spherical direction. Note that $\tilde{\mathcal{Q}}(u,\na u,\na^2 u)$ is analytic with respect to $(u,\na u,\na^2 u)$ and consists of terms that are cubic and higher power in $(u,\na u,\na^2 u)$ which satisfies
\begin{align*}
    \tilde{\mathcal{Q}}(u,\na u,\na^2 u)\leq  C(|\na u|^4+|\na u|^2|\na^2u|)
\end{align*}

First we prove that the estimates for higher order derivatives are guaranteed by smallness of $C^1$-norm.

\begin{Lem}
    Suppose $u$ be a solution to \eqref{e:Q} satisfying $|u|_{C^1(\mathcal{C}_k \times [0,\infty))} \le \ep$. Then for any $p>0$ we have
    \begin{equation}\label{e:assume C^p-norm small}
    |u|_{C^p(\mathcal{C}_k \times [1,\infty))} \le C(n,p)\ep,
\end{equation}
\end{Lem}
\begin{proof}
    Given $t_0$. We consider the rescaled function 
    \begin{align*}
        v_{t_0}(y,\theta,t)=u(e^{\frac{t}{2}}y,\theta,t+{t_0}).
    \end{align*}

Then we know that 
\begin{align*}
    \partial_tv_{t_0}&=(\na u\cdot \frac{x}{2})(e^{\frac{t}{2}}y,\theta,t)+\partial_tu(e^{\frac{t}{2}}y,\theta,t+{t_0})\\ \notag&=(\na u\cdot \frac{x}{2})(e^{\frac{t}{2}}y,\theta,t+t_0)+(Lu+\mathcal{Q}(u))(e^{\frac{t}{2}}y,\theta,t+{t_0})\\ \notag&=(\Delta u+u+\mathcal{Q}(u))(e^{\frac{t}{2}}y,\theta,t+{t_0})\\ \notag &=v+(e^{-t}\Delta_y+\Delta_\theta)v+ \mathcal{Q}(v,e^{-\frac{t}{2}}\na_y v,\na_\theta v,e^{-t}\na^2_{yy} v,e^{-\frac{t}{2}}\na^2_{\theta y} v,\na^2_{\theta\theta} v).
\end{align*}

Applying local derivative estimates (see \cite{LZ2}*{Proposition 4.1}, \cite{Sto}*{Proposition C.3} and also \cites{Bam,App}) we have that  
\begin{align*}
    |v_{t_0}|_{C^p(\mathcal{C}_k \times [-\frac{1}{2},\frac{1}{2}])}\leq C(n,p) |v_{t_0}|_{C^1(\mathcal{C}_k \times [-1,1])}.
\end{align*}
Since $t_0$ is arbitrary we have that 
\begin{equation*}
        |u|_{C^p(\mathcal{C}_k \times [1,\infty))} \le C(n,p)\ep. 
    \end{equation*}
\end{proof}

Let $\mu$ be the standard measure on the cylinder $\mathcal{C}_k = \RR^k \times \mathbb{S}^{n-k}_{\sqrt{2(n-k)}}$. We denote by $L^2$ be the $L^2$-space of functions on $\mathcal{C}_k$ with respect to the weighted $L^2$ measure $e^{-\frac{|x|^2}{4}}d\mu$.

For $p < \infty$, the $H^p$-norm is defined in a standard way as follows
\begin{equation*}
    |u|^2_{H^p} \equiv \sum_{2k + l \le p} \int |\nabla^l \partial_t^k  u|^2  e^{-\frac{|x|^2}{4}} d\mu.
\end{equation*}

First we prove that $|u|_{H^p}$ controls $|Q|_{H^{p}}$ for $p\ge 2$. This can be seen from the quadratic nature of $Q$. The proof follows Lemma 6.2 in \cite{CH24}, where similar result was proved for the compact case.

\begin{Lem}\label{l:H^p (u) controls Q}
    Let $p\ge2$ and $u$ be a solution to \eqref{e:Q} satisfying $|u|_{C^p} \le \ep$. Then we have
    \begin{equation*}
        |\mathcal{Q}|_{H^p} \le C(n,p,\mathcal{Q})\ep |u|_{H^p}. 
    \end{equation*}
\end{Lem}

\begin{proof}
     
     From \eqref{e:Q} we know that Q satisfies \begin{align*}
    \mathcal{Q}(u)=\sum_{i=0}^2q_{i}\cdot \na^i u, 
\end{align*}
where $q_i(u,\na u)$ depend only on $u,\,\na u$ and are independent with respect to the location $x$. Here $q_i$ are smooth and satisfy that \begin{align*}
    q_i(0,0)=0.
\end{align*}
Combining with \eqref{e:Q}, we have that 
\begin{align*}
    |q_i(u,\na u)|\leq C(|u|+|\na u|+|\na u|^3).
\end{align*}
Also by taking derivatives, we can get
\begin{align*}
   |q_i|_{C^k}\leq C(k)(|u|_{C^{k+1}}+ |u|_{C^{k+1}}^3).
\end{align*}
The rest follows from Leibniz rule, similar as \cite{CH24}*{Lemma 6.2}
\end{proof}

We define another norm $V$ as follows:
\begin{equation*}
    |u|^2_{V^0}\equiv |u|^2_{L^2} = \int_{\mathcal{C}_k} u^2 e^{-\frac{|x|^2}{4}} d\mu,
\end{equation*}
and
\begin{equation*}
    |u|_{V^p}^2 \equiv  \sum_{k+l\le p} \int |L^l \partial_t^k u|^2 e^{-\frac{|x|^2}{4}} d\mu.
\end{equation*}

Then main result of this section is to prove the following estimate for $\mathcal{Q}$ similar to Lemma \ref{l:H^p (u) controls Q}. To prove this Proposition, we will show that the $V^p$-norm is equivalent to $H^{2p}$-norm. The proof of compact case can be found in \cite{CH24}*{Lemma 6.2}.

\begin{Pro}\label{p:u controls Quadratic term in V-norm}
Let $p\ge 1$ and $u$ be a solution to \eqref{e:Q} satisfying $|u|_{C^1} \le \ep$. Then we have
    \begin{equation*}
        |\mathcal{Q}|_{V^p} \le  \ep(n,p)\cdot |u|_{V^p}.
    \end{equation*}
\end{Pro}

\begin{proof}[Proof of Proposition \ref{p:u controls Quadratic term in V-norm}]
    This is a direct consequence of Lemma \ref{l:H^p (u) controls Q} and the following Lemma \ref{l:V-norm equiv to H-norm}. 
\end{proof}

We prove the equivalence of $V^p$-norm and $H^{2p}$-norm.

\begin{Lem}\label{l:V-norm equiv to H-norm}
    Let $p\ge 0$. There exists some $C(n,p)>0$ such that for any $u \in H^{2p} \cap V^{p}$
    \begin{equation*}
        C^{-1} |u|^2_{V^p} \le |u|^2_{H^{2p}} \le C|u|^2_{V^p}.
    \end{equation*}
\end{Lem}

\begin{proof}
    First we recall that (see \cite{CM15} Lemma 3.11), for any $u \in H^2$,
    \begin{equation}\label{e:CM V-norm = H-norm}
        |Lu|_{L^2} \le c (|u|_{L^2} + |\nabla u|_{L^2} + |\nabla^2 u|_{L^2}) \le c'( |Lu|_{L^2} + |u|_{L^2} ).
    \end{equation}
Hence, the lemma holds for $p=1$.

We will use this inequality to prove the lemma. 

\textbf{Left Hand Side: }
We prove it by induction. The case $p=1$ holds by \eqref{e:CM V-norm = H-norm}. Suppose that the result holds for $p=p_0 \ge 1$. Then by induction we have
\begin{align*}
    |u|^2_{V^{p_0+1}} \le C |u|^2_{H^{2 p_0}}  + \sum_{k+l=p_0 +1 } \int |L^l \partial_t^k u|^2 e^{-\frac{|x|^2}{4}} d\mu 
\end{align*}

Fix $k+l = p_0 +1$. If $l\ge 1$, let $w \equiv L^{l-1} \partial_t^k u$. Then by \eqref{e:CM V-norm = H-norm} and induction hypothesis,
\begin{equation*}
    |L^l \partial_t^k u|_{L^2} \le |L u|_{V^{p_0}} \le C
    |L u|_{H^{2p_0}} \le C| u|_{H^{2p_0 + 2}}.
\end{equation*}
If $k \ge 1$, by induction hypothesis,
\begin{equation*}
    |L^{l} \partial_t^{k} u|_{L^2} \le |\partial_t u|_{V^{p_0}} \le C |\partial_t u|_{H^{2p_0}}  
    \le C|u|_{H^{2p_0 +2}}.
\end{equation*}

Combining all, we finish the induction:
\begin{equation*}
    |u|^2_{V^{p_0+1}} \le C|u|^2_{H^{2p_0+2}}.
\end{equation*}

\textbf{Right Hand Side: } Similarly, we note that the inequality holds for $p=1$ by \eqref{e:CM V-norm = H-norm} and we assume that it holds for $p=p_0 \ge 1$. That is
\begin{equation*}
    |u|^2_{H^{2p_0}} \le C|u|^2_{V^{p_0}}. 
\end{equation*}

Then
\begin{align*}
    |u|^2_{H^{2p_0+2}} = |u|^2_{H^{2p_0}} + \sum_{2p_0+1 \le 2k +l \le 2p_0} |\nabla^l \partial_t^k u|^2_{L^2} \le |u|^2_{V^{p_0}} + \sum_{2p_0+1 \le 2k +l \le 2p_0} |\nabla^l \partial_t^k u|^2_{L^2}.
\end{align*}

Next, we deal with the last term. Suppose $k\ge 1$. Hence by the induction hypothesis
\begin{align*}
    |\nabla^l \partial_t^k u|^2_{L^2} \le |\partial_t u|^2_{H^{2 p_0}} \le C|\partial_t u|^2_{V^{p_0}}  \le C|u|^2 _{V^{p_0+1}}.
\end{align*}

For the case where  $k=0$ and $l=p_0+1 \ge 2$, since we have that for $\Delta_f u=\Delta u-\la \na f,\,\na u\ra$, $\text{Ric}_f=\text{Ric}+\na^2f$ and $f=\frac{|x|^2}{4}$, by Bochner identity we have
\begin{align*}
    \Delta_f\na T=\na \Delta_f T+\text{Ric}_f(\na T).
\end{align*}
So we would have 
\begin{align*}
    |\nabla^l  u|^2_{L^2}&=\int\la \na^{l-1}u,\,\Delta_f\na^{l-1}u\ra d\mu \le \int\la \na^{l-1}u,\,\na^{l-1}\Delta_fu\ra d\mu +C| u|^2_{H^{2 p_0}} \\ &\le \int\la \Delta_f\na^{l-2}u,\,\na^{l-2}\Delta_fu\ra d\mu +C| u|^2_{H^{2 p_0}} \\ &\le \int\la \na^{l-2}\Delta_fu,\,\na^{l-2}\Delta_fu\ra d\mu +C| u|^2_{H^{2 p_0}} \le |L u|^2_{V^{p_0-1}}+ |u|^2_{V^{p_0}} \le C|u|^2 _{V^{p_0+1}}.
\end{align*}

Hence we finish the induction step. Then we prove the both sides and the proof is now finished.
\end{proof}

Let us end this section by a useful Lemma.
\begin{Lem}\label{l:polynomial to derivative}
    Let $u \in C^{\infty}(\mathcal{C}_k) \cap H^p (\mathcal{C}_k)$ be a solution to \eqref{e:Q}. Then
    \begin{equation*}
        \int_{\mathcal{C}_k} u |x|^p  e^{-\frac{|x|^2}{4}} d\mu \le C(n,p) |u|_{H^p} .
    \end{equation*}
\end{Lem}

\begin{proof}
    Using integration by parts, we have
    \begin{align*}
        0 = \int_{\mathcal{C}_k} \partial_i( u e^{-\frac{|x|^2}{4}}) d\mu = \int_{\mathcal{C}_k} (\partial_i u)  e^{-\frac{|x|^2}{4}} d\mu - \int_{\mathcal{C}_k}  \frac{u x_i}{2}  e^{-\frac{|x|^2}{4}} d\mu.
    \end{align*}
    By H\"older inequality,
    \begin{equation*}
        \int_{\mathcal{C}_k}  u x_i  e^{-\frac{|x|^2}{4}} d\mu = 2 \int_{\mathcal{C}_k} (\partial_i u)  e^{-\frac{|x|^2}{4}} d\mu \le C |u|_{H^1}.  
    \end{equation*}
    Then by iterations,
    \begin{align*}
        \int_{\mathcal{C}_k} u |x|^p  e^{-\frac{|x|^2}{4}} d\mu \le C(k,p) |u|_{H^p} .
    \end{align*}
\end{proof}

\section{Unique continuation of global graphs}

In this section, we prove the main Theorems \ref{t:main} and \ref{t:local}. Recall that we can write the rescaled mean curvature flow equation over the cylinder $\mathcal{C}_k$ as
\begin{equation*}
    \partial_tu=Lu+\mathcal{Q}(u,\na u,\na^2u)
\end{equation*}
where $L$ is the linearized operator
\begin{align*}
    Lu=\Delta u-\frac{x}{2}\cdot \na u+u,
\end{align*}
with eigenvalues
\begin{equation*}
    \lambda_1 > \lambda_2 \ge ... \to -\infty. 
\end{equation*}

A key observation is that we can find a positive gap for the difference of two different eigenvalues, which is crucial for the proof of maim theorems.

\begin{Lem}\label{l:sepctral gap}
    Set $ b_\lambda=\min\{|\lambda-\lambda_i|\,:\,\lambda_i\neq \lambda\}$. Then there exists some $c_0(n)>0$ such that $b_{\lambda} > c_0$ for any eigenvalue $\lambda$. 
\end{Lem}

\begin{proof}
    We can write the operator $L$ as 
    \begin{equation*}
        L  = (\Delta_{\RR^k}  - \frac{1}{2} x \cdot \nabla_{\RR^k} ) + \Delta_{\mathbb{S}^{n-k}} + 1.
    \end{equation*}
    Hence the spectrum of $L$ on $\mathbb{S}^{n-k}_{\sqrt{2(n-k)}}\times \RR^k$ is 
    \begin{equation*}
        \bigg\{-\frac{j}{2} - \frac{i(i-1+n-k)}{2(n-k)}+1 : i, j \in \mathbb{N} \bigg\}.
    \end{equation*}
    Since $n,k$ are fixed, it implies that $b_{\lambda} > c_0(n)$ for any eigenvalue $\lambda$. 
\end{proof}

We use $\phi_i$ to denote the eigenfunction of $L$ corresponding to eigenvalue $\lambda_i$. 
We set
\begin{equation*}
    u_i(t) \equiv \int u \cdot \phi_i e^{-\frac{|x|^2}{4}}d\mu,
\end{equation*}
and for $1\le j \le j_0$
\begin{equation*}
    z_j(t) \equiv u_{i_0+j}(t),
\end{equation*}
where $\lambda_{i_0+j}=0$ corresponds to the zero eigenvalue.

We further define: 
\begin{equation*}
    u^{k,l}_i \equiv \int L^l \partial^k_t u \cdot \phi_i  e^{-\frac{|x|^2}{4}} d\mu, \quad \mathcal{Q}^{k,l}_i \equiv \int  L^l \partial_t^k \mathcal{Q}(u) \cdot \phi_i  e^{-\frac{|x|^2}{4}}d\mu.
\end{equation*}
and
\begin{align}\label{general x comp}
        x_{\lambda,0}(t) = \Big( \sum_{k+ l \le p} \sum_{\lambda_i = \lambda} |u^{k,l}_i| ^2 (t)\Big)^{1/2}, \quad x_{\lambda,\pm}(t) = \Big( \sum_{k+l\le p} \sum_{\pm (\lambda_i- \lambda)>0}  |u^{k,l}_i|^2 (t) \Big)^{1/2},
\end{align}
and thus we have
\begin{equation*}
    |u|_{V^p}^2 = x_{\lambda,0}^2 + x_{\lambda,+}^2 + x_{\lambda,-}^2.
\end{equation*}

By \eqref{e:Q} we have the following equations
\begin{align}\label{e:compe}
    \frac{d}{dt} u^{k,l}_i - \lambda_i u^{k,l}_i = \mathcal{Q}^{k,l}_i. 
\end{align}

\subsection{Global super-exponential convergence}

\begin{Thm}[Theorem \ref{t:main}]
     Let $\tilde{\varphi}:M\times[0,\infty)\to \RR^{n+1}$ be a solution to the rescaled mean curvature flow equation \eqref{e:recaled mcf} such that $M_t \equiv \tilde{\varphi}(M,t)$ can be expressed as a normal graph of a function $u(x,t)$ over the round cylinder $\mathcal{C}_k=\mathbb{S}^{n-k}_{\sqrt{2(n-k)}}\times \RR^{k}$. Then there exists some $\ep_0(n)$ such that the following holds: if $\|u\|_{C^1(\mathcal{C}_k \times [0, \infty))}\leq \ep_0$ and we have
    \begin{equation}\label{e:growthcontrol}
    \lim_{t\to \infty} e^{\alpha t}\int_{\mathcal{C}_k} u^2(x,t) e^{-\frac{|x|^2}{4}}d\mu =0 \text{ for any } \alpha\in \mathbb{N}
\end{equation}   
where $d\mu$ is the standard measure on $\mathcal{C}_k$, then we have $u \equiv 0$.
\end{Thm}

Let us start with following Lemma, which shows that the assumptions in Theorem \ref{t:main} guaranteed convergence in $C^\infty_{\text{loc}}$ and super-exponential convergence in $H^k(e^{-\frac{|x|^2}{4}}d\mu)$.

\begin{Lem}\label{l:improve}
  For the $u$ defined in Theorem \ref{t:main}, we have that $u\to 0$ in $C^\infty_{\text{loc}}$ and
 \begin{equation}\label{e:highcon}
    \lim_{t\to \infty} e^{\alpha t}\int_{\mathcal{C}_k} |\na ^lu|^2(x,t) e^{-\frac{|x|^2}{4}}d\mu =0 \text{ for any } \alpha,l\in \mathbb{N}.
\end{equation}   
\end{Lem}
\begin{proof}[Proof of Lemma \ref{l:improve}]
From Gagliardo-Nirenberg interpolation inequality (c.f. \cite{CM15} \cite{Z20a}*{(A,1),(A.2)}) combining with \eqref{e:assume C^p-norm small}, we have that 
\begin{align*}
|\na^lu(x)|\leq C(|x|+1)^{l+n}e^{\frac{1-\delta(l,n)}{4}|x|^2}(\int_{\mathcal{C}_k} u^2(x,t) e^{-\frac{|x|^2}{4}}d\mu)^{C_2}\\
    \int_{\mathcal{C}_k} |\na ^lu|^2(x,t) e^{-\frac{|x|^2}{4}}d\mu\leq C_1 (\int_{\mathcal{C}_k} u^2(x,t) e^{-\frac{|x|^2}{4}}d\mu)^{C_2}.
\end{align*}
Thus, combining with \eqref{e:growthcontrol} we could get  $u\to 0$ in $C^\infty_{\text{loc}}$ and \eqref{e:highcon}.
\end{proof}

\begin{proof}[Proof of Theorem \ref{t:main}]

Recall that $b_\lambda=\min\{|\lambda-\lambda_i|\,:\,\lambda_i\neq \lambda\}$. 

From \eqref{e:compe}, we have 
\begin{align*}
    \frac{d}{dt}(|u_i^{k,l}|^2)=2u_i^{k,l}(\lambda_i u^{k,l}_i + \mathcal{Q}^{k,l}_i)=2\lambda_i|u_i^{k,l}|^2+2u_i^{k,l} \mathcal{Q}^{k,l}_i,
\end{align*}
and by Cauchy-Schwarz inequality, we have
\begin{align*}
    x_{\lambda,+}'- \lambda x_{\lambda,+} &\ge b_\lambda   x_{\lambda,+}  - |\mathcal{Q}|_{V^p},\\
    x_{\lambda,-}'-\lambda x_{\lambda,-} &\le -b_\lambda x_{\lambda,-} + |\mathcal{Q}|_{V^p},\\
    |x_{\lambda,0}'-\lambda x_{\lambda,0}| &\le |Q|_{V^p}. 
\end{align*}
Note that by Lemma \ref{l:H^p (u) controls Q}, we have 
\begin{align*}
    |\mathcal{Q}|_{V^p}\leq C\ep_0\cdot |u|_{V^p} \le \ep(x_{\lambda,+}+x_{\lambda,0}+x_{\lambda,-}).
\end{align*}
where denote $\ep:=C\ep_0$. So if we take
\begin{align*}
    \tilde{x}_{\lambda,+}=&e^{-\lambda t}{x}_{\lambda,+},\\
        \tilde{x}_{\lambda,0}=&e^{-\lambda t}{x}_{\lambda},\\    \tilde{x}_{\lambda,-}=&e^{-\lambda t}{x}_{\lambda,-},
\end{align*}
and by direct computation we obtain
\begin{align*}
    \tilde{x}_{\lambda,+}'\geq& b_\lambda\tilde{x}_{\lambda,+}-\ep(\tilde{x}_{\lambda,+}+\tilde{x}_{\lambda,0}+\tilde{x}_{\lambda,-}),\\
      \tilde{x}_{\lambda,-}'\leq &-b_\lambda\tilde{x}_{\lambda,-}+\ep(\tilde{x}_{\lambda,+}+\tilde{x}_{\lambda,0}+\tilde{x}_{\lambda,-}),\\
        |\tilde{x}_{\lambda,0}'|\leq &\ep(\tilde{x}_{\lambda,+}+\tilde{x}_{\lambda,0}+\tilde{x}_{\lambda,-})
\end{align*}

By assumption and Lemma \ref{l:improve}, we know $\liminf_{t \to \infty} \tilde{x}_{\lambda,+} = 0$. From Lemma \ref{L:Merle-Zaag} and the gap \ref{l:sepctral gap}, after multiplying $e^{\lambda t}$ we can get 
\begin{align}\label{e:smallp}
     {x}_{\lambda,+}\leq 2\frac{\ep}{b_\lambda} ( {x}_{\lambda,0}+ {x}_{\lambda,-}).
\end{align}
Moreover, one the the following holds:
\begin{align}\label{e:zerodom}
{x}_{\lambda,-}\leq 8\frac{\ep}{b_\lambda}  {x}_{\lambda,0}\,\,\text{in}\,\,[s_\lambda,\infty)
\end{align}
or\begin{align}\label{e:negdom}
    {x}_{\lambda,0}\leq 20\frac{\ep}{b_\lambda}  {x}_{\lambda,-}\,\,\text{in}\,\,[0,\infty).
\end{align}
We claim that equation \eqref{e:negdom} holds. Let us prove by contradiction. Suppose that equation \eqref{e:zerodom} holds, then let us set \begin{align*}
    z_{\lambda,i}(t)=\int u \phi_{\lambda,i}e^{-\frac{|x|^2}{4}}d\mu,
\end{align*}
where $ \phi_{\lambda,i}$ is the $i$-th eigenfunction corresponding to eigenvalue $\lambda$. Then we claim that \begin{align}\label{e:zerocontrolhigh}
    |x_{\lambda,0}|\leq C(\lambda,p)|z_\lambda|.
\end{align}
In the following we suppose $i \in I_\lambda$, which means $\lambda_i = \lambda$. Note that
\begin{equation*}
    |x_{\lambda,0}|^2 \le |z_\lambda|^2 + \sum_{i \in I_\lambda} \Bigg( \sum_{\substack{k+l \le p\\ l \ge 1}} |u^{k,l}_i|^2 + \sum_{\substack{k+l\le p \\ k\ge1}} |u^{k,l}_i|^2 \Bigg).
\end{equation*}

For $l \ge 1$, since $L \phi_{\lambda,i} = \lambda\phi_{\lambda,i}$, then
\begin{align*}
    u^{k,l}_i = \int \langle L^l \partial_t^k u, \phi_i \rangle d\mu = \int \langle L^{l-1} \partial_t^k u, L \phi_i \rangle  d \mu = \lambda u^{k,l-1}_i
\end{align*}

For $k\ge 1$, we have
\begin{align*}
    u^{k,l}_i = \frac{d}{dt} u^{k-1,l}_i =\lambda u^{k-1,l}+ \mathcal{Q}^{k-1,l}_i.
\end{align*}

Then by induction we would have 
\begin{align*}
u^{k,l}_i=\lambda^{l}u_i^{k,0}=\lambda^l(\sum_{m=1}^{k-1} \lambda^{m-1}\mathcal{Q}_i^{k-m,0}+\lambda^k u_i^{0,0}),    
\end{align*}
Note  that we have \begin{align*}
    \lambda^l \lambda^{m-1}\mathcal{Q}_i^{k-m,0}=\mathcal{Q}^{k-m,m-1+l}_i.
\end{align*}

Combining with Lemma \ref{p:u controls Quadratic term in V-norm}, we would have 
\begin{align*}
    |x_{\lambda,0}|^2 \le C(\lambda,p)|z_\lambda|^2 + C(p)\sum_{i \in I_\lambda} \sum_{k+l \le p-1} |\mathcal{Q}^{k,l}_i|^2 \le C(\lambda,p)|z_\lambda|^2 + C(p)\frac{\ep}{b_\lambda} |u|_{V^p}^2 \le C(\lambda,p)|z_\lambda|^2 + 2 \frac{\ep}{b_\lambda} C(p)|x_{\lambda,0}|^2.
\end{align*}
Since $p$ is chosen initially which is independent with respect to $\lambda$, we could take $\ep $ sufficiently small such that $2 \frac{\ep}{b_\lambda}C(p)\leq \frac{1}{2}$, where we used that $b_\lambda>c_0$ for all $\lambda$ corresponding to the spectrum gap, and thus the claim is proved, that is, \begin{align}
    |x_{\lambda,0}|\leq C(\lambda,p)|z_\lambda|.\label{e:zdom}
\end{align}

Then we have 
\begin{align*}
    |z_\lambda'|z_\lambda=\frac{1}{2}|(z_\lambda^2)'|\leq \sum_{\lambda_i=\lambda}|u_i||u_i'|\leq z_\lambda(\sum_{\lambda_i=\lambda}|u_i'|^2)^{\frac{1}{2}}.
\end{align*}

Now we would like to control $(\sum_{\lambda_i=\lambda}|u_i'|^2)^{\frac{1}{2}}$. We have
\begin{align*}
u'_i = \langle u', \phi_i \rangle  =  \langle Lu + \mathcal{Q}u, \phi_i \rangle = \lambda u_i+\int \mathcal{Q} \cdot \phi_i e^{-\frac{|x|^2}{4}}d\mu ,
\end{align*}
and 
\begin{align*}
    \sum_{\lambda_i=\lambda}|u_i'|^2&=  \sum_{\lambda_i=\lambda}( \lambda u_i+\int \mathcal{Q} \cdot \phi_i e^{-\frac{|x|^2}{4}}d\mu)^2\leq\sum_{\lambda_i=\lambda} 2\lambda^2u_i^2+2(\int \mathcal{Q} \cdot \phi_i e^{-\frac{|x|^2}{4}}d\mu)^2\\ &\leq 2\lambda^2|z_\lambda|^2+2\int \mathcal{Q}^2 e^{-\frac{|x|^2}{4}}d\mu\leq C(n,\lambda,\ep)|u|_{H^2}^2.
\end{align*}
From \eqref{e:smallp} and \eqref{e:zerodom} we conclude that
\begin{equation}\label{e:growthcon}
    |z_\lambda'| \le C(n,\lambda,\ep) z_\lambda.
\end{equation}

So $|\ln(z_\lambda)'|\le C(n,\lambda,\ep) $ and thus $z_\lambda(t)\geq z_\lambda(s_\lambda)e^{-C(n,\lambda,\ep)(t-s_\lambda)}$. But this contradicts with \eqref{e:growthcontrol}. So we know that \eqref{e:zerodom} is false and thus \eqref{e:negdom} is true for all $\lambda$. So by sending $\lambda \to - \infty$, we would know that
\begin{align*}
    \|u\|_{H^p}\leq \lim_{i\to\infty}\Big( \sum_{k+l\le p} \sum_{ \lambda_j >\lambda_i}  |u^{k,l}_j|^2 (t) \Big)^{1/2}=0.
\end{align*}
This implies that $u \equiv 0$ and hence completes the proof. 
\end{proof}

\subsection{Local super-exponential convergence}

We prove the uniform upper bound for the convergence rate of $u$ in any region $\Omega \subset \mathcal{C}_k$. 

\begin{thm}[Theorem \ref{t:local}]
     Let $\tilde{\varphi}:M\times[0,\infty)\to \RR^{n+1}$ be a solution to the rescaled mean curvature flow equation \eqref{e:recaled mcf} and $M_t$ can be expressed as a normal graph of a function $u(p,t)$ over $\mathcal{C}_k=\mathbb{S}^{n-k}_{\sqrt{2(n-k)}}\times \RR^{n-k}$ such that $\lim\limits_{t\to \infty}\|u(\cdot,t)\|_{C^1(\mathcal{C}_k)}= 0$. If $u\not\equiv 0$, then there exists an $\alpha_0$  such that for every open set $\Omega\subset \mathcal{C}_k$
\begin{align*}
    \int_{\Omega}u^2(x,t) \,e^{-\frac{|x|^2}{4}}d\mu \geq C(\Omega,u)e^{-\alpha_0t}.
\end{align*}
where $d\mu$ is the standard measure on $\mathcal{C}_k$. 
\end{thm}

\begin{proof}

As we have that $\lim\limits_{t\to \infty}\|u(\cdot,t)\|_{C^1(\mathcal{C}_k)}= 0$, combining with Lemma \ref{l:improve} we could know that for $\forall l\in\mathbb{N}$
\begin{align*}
\lim_{t\to\infty}\int_{\mathcal{C}_k}|\na^lu|^2(x,t) e^{-\frac{|x|^2}{4}}d\mu=0.
\end{align*}

As we assume $u\not\equiv 0$, then by Theorem \ref{t:main}, there exists some fixed eigenvalue $\lambda_m \le 0$ such that the following holds:\begin{align}
     \limsup_{s\to \infty} e^{-\lambda_ms} \|u(\cdot,s)\|_{H^{2p}}&=0\\
      \limsup_{s\to \infty} e^{-\lambda_{m+1}s} \|u(\cdot,s)\|_{H^{2p}}&\neq 0.\label{e:lowbound}
\end{align}

Now let us use the proof of \ref{t:main} to see that either
\begin{align}
     &|u|_{H^{2p}} \le C(m,p) |z_{\lambda_m}|, \label{e:first}\\
     & |z_{\lambda_m}(t)| \ge Ce^{-C(n,\lambda_m,\ep)t} \quad\text{for}\,\,t \label{e:largecomp}\,\,\text{sufficiently large},
\end{align}
or
\begin{align}\label{e:sec1}
    &\limsup_{s\to \infty} e^{(-\lambda_{m+1}- \ep)s} \|u(\cdot,s)\|_{H^{2p}} =  0,\\
    &|u|_{H^{2p}} \le C(m,p) |x_{\lambda_{m},-}|.\label{e:sec2}
\end{align}

\textbf{Case 1: \eqref{e:first} \eqref{e:largecomp} are true.}

We denote the projection of $u$ on the eigenspace corresponding to $\lambda_m$  of $L$ as $ u_0$ and denote $u_1=u-u_0$.

By the assumption, combining with \eqref{e:smallp}, \eqref{e:zerodom}, \eqref{e:zerocontrolhigh} and $u\to 0$ in $C^1$, we have that for any $\ep>0$, there exists a time $t(\ep,\lambda_m)$ such that for all $t>t(\ep,\lambda_m)$
\begin{align*}
    \|u_1\|_{L^2(\mathcal{C}_k)}\leq \ep \|u_0\|_{L^2(\mathcal{C}_k)}.
\end{align*}

On the other hand, since the eigenspace corresponding to $\lambda_m$ is a finite-dimensional space and eigen functions of $L$ have unique continuation properties. By a contradiction argument, we have that  $\|\cdot\|_{L^2(\mathcal{C}_k)}$, $\|\cdot\|_{L^2(\Omega)}$ are both norms. Thus, we know that  
\begin{align*}
    \|u_0\|_{L^2(\Omega)}\geq \delta(\Omega,k,m,n)\cdot\|u_0\|_{L^2(\mathcal{C}_k)}.
\end{align*}

Thus we have
\begin{align*}
    \int_{\Omega}u^2(x,t) e^{-\frac{|x|^2}{4}}d\mu=\int_{\Omega}(u_0^2+2u_0u_1+u_1^2) e^{-\frac{|x|^2}{4}}d\mu \geq \|u_0\|^2_{L^2(\mathcal{C}_k)}(\delta(\Omega,k,m,n)-2\ep).
\end{align*}
 So, if we choose $\ep$ to be small enough such that $\ep<\frac{\delta(\Omega,k,m,n)}{100}$, combining with \eqref{e:growthcon} \eqref{e:largecomp} we would have for $t>t(\ep)$
\begin{align*}
    \|u\|_{L^2(\Omega)}\geq{C(u,m,p)\delta(\Omega,k,m,n)}e^{-C(n,\lambda_m,\ep)t}.
\end{align*}
Thus, in this case, Theorem \ref{t:local} is proved.

 \textbf{Case 2: \eqref{e:sec1} \eqref{e:sec2} are true.}

 Now we only need to deal with the second case. If we take
\begin{align*}
    \tilde{x}_{\lambda_{m+1},\ep,+}=&e^{-(\lambda_{m+1}+\ep) t}{x}_{\lambda_{k+1},+},\\
        \tilde{x}_{\lambda_{m+1},\ep,0}=&e^{-(\lambda_{m+1}+\ep) t}{x}_{\lambda_{m+1},0},\\    \tilde{x}_{\lambda_{m+1},\ep,-}=&e^{-(\lambda_{m+1}+\ep)  t}{x}_{\lambda_{m+1},-},
\end{align*}
we could get for $t\geq t(\ep)$
\begin{align*}
    \tilde{x}_{\lambda_{m+1},\ep,+}'\geq& (b_{\lambda_{m+1}}-\ep)\tilde{x}_{\lambda_{m+1},\ep,+}-\ep(\tilde{x}_{\lambda_{m+1},\ep,+}+\tilde{x}_{\lambda_{m+1},\ep,0}+\tilde{x}_{\lambda_{m+1},\ep,-}),\\
      \tilde{x}_{\lambda_{m+1},\ep,-}'\leq &-(b_{\lambda_{m+1}}+\ep)\tilde{x}_{\lambda_{m+1},\ep,-}+\ep(\tilde{x}_{\lambda_{m+1},\ep,+}+\tilde{x}_{\lambda_{m+1},\ep,0}+\tilde{x}_{\lambda_{m+1},\ep,-}),\\
        |\tilde{x}_{\lambda_{m+1},\ep,0}'|\leq &2\ep(\tilde{x}_{\lambda_{m+1},\ep,+}+\tilde{x}_{\lambda_{m+1},\ep,0}+\tilde{x}_{\lambda_{m+1},\ep,-})
\end{align*}
where $b_{\lambda_{m+1}}=\min\{|\lambda-\lambda_{m+1}|\,|\,\text{eigenvalue}\,\,\lambda\neq \lambda_{m+1}\}>c_0$.

As we have that \begin{align*}
    &\limsup_{s\to \infty} e^{(-\lambda_{m+1}- \ep)s} \|u(\cdot,s)\|_{H^{2p}} =  0,
\end{align*}which implies that 
\begin{align*}
    \limsup_{s\to \infty}\tilde{x}_{\lambda_{m+1},\ep,+}=0.
\end{align*} 
we are allowed to use the ODE Lemma \ref{L:Merle-Zaag} again to obtain that either\begin{align*}
     &|u|_{H^{2p}} \le C(m+1,p) |z_{\lambda_{m+1}}|, \\
     & |z_{\lambda_{m+1}}(t)| \ge Ce^{-C(n,\lambda_{m+1},2\ep)t} \quad\text{for}\,\,t \,\,\text{sufficiently large},
\end{align*}
or
\begin{align*}
    &\limsup_{s\to \infty} e^{(-\lambda_{m+2}- 2\ep)s} \|u(\cdot,s)\|_{H^{2p}} =  0,\\
    &|u|_{H^{2p}} \le C(m+1,p) |x_{\lambda_{m+1},-}|.
\end{align*}  
Using the same argument in previous discussions, the first case is not possible. Hence the only possibility is 
\begin{align}
    &\limsup_{s\to \infty} e^{(-\lambda_{m+2}- 2\ep)s} \|u(\cdot,s)\|_{H^{2p}} =  0,\label{e:control}
\end{align}
However, by  taking $\ep\leq \frac{c_0}{10} \le\frac{1}{10}(\lambda_m -\lambda_{m+1})$, \eqref{e:control} would contradict with  \eqref{e:lowbound}. Thus Case 2 is not possible. As we already proved Theorem \ref{t:local} in Case 1, we have finished the proof.

\end{proof}

\section{Counter-examples}

In this section, we construct counterexamples
to the unique continuation problem for the local solutions to the nonlinear parabolic equations.
Our starting point is Tikhonov’s classical counterexample for the heat equation,
which we extend to a nonlinear setting by constructing an analytic family of
solutions via power series expansion.

The underlying difficulty lies in the nonlinear structure of the equation and the ill-posed-ness of the backward equation.
Unlike the linear heat equation, for which Tikhonov’s classical construction
already reveals the failure of backward uniqueness, the nonlinear backward equation introduces strong coupling between the derivatives of $u$,
and no integral representation such as the Duhamel formula can be applied directly.
As a result, we construct solutions by hand, using formal expansions and delicate control of their convergence.

Our approach is based on an power-series construction
combined with a quantitative combinatorial analysis. The nonlinearity creates a hierarchy of nonlinear interactions among the coefficients
in the series, and bounding their growth requires precise counting arguments. We show that, thanks to the quadratic nature of the nonlinear term $\mathcal{Q}(u,\nabla u, D^2u)$, the combinatorial complexity can be dominated by the super-exponential decay
inherited from the base Tikhonov solution of the heat equation. This ensures that the formal series converges, yielding a genuine smooth solution
with the prescribed super-exponential behavior.

For clarity of exposition, we first review the Tikhonov's examples for solutions to heat equation with super-exponential decay and generalize it to arbitrary decay rate. Then we study the nonlinear equation. We first present the proof in the simplest case, where the nonlinearity is quadratic, and then consider the case where the non-linear term only depends on $u'$. We will prove the general case in the next section,

\subsection{Tikhonov's counter-examples}\label{subsection:T}

First we review the examples originally due to Tikhonov \cite{T35}. 
Define a function $\varphi$ on $\RR_+$ as follows:
\[
\varphi(t) = 
\begin{cases}
    e^{-1/t^2}, &t \ge 0 \\
    0, &t <0.
\end{cases}
\]

Then we define 
\begin{equation}\label{e:Tikhonov example}
    u_0(x,t) = \sum_{k=0}^{\infty} \frac{\varphi^{(k)}(t) x^{2k}}{(2k)!}
\end{equation}

It is easy to check that $u$ satisfies the heat equation with $u(\cdot,0) \equiv 0$. Moreover, using Cauchy's integral formula for analytic functions as in \cite{J71}, for any $\ep>0$, there exists some $C_{\ep}>0$ such that
\begin{equation}\label{e:estimate for Tik0}
    |\varphi^{(k)}(t)| \le C_{\ep} \frac{k!}{(\ep t)^k} e^{-\frac{\ep}{t^2}}.
\end{equation}

This implies that
\begin{equation*}
    |u_0(x,t)| \le C_{\ep} e^{-\frac{\ep}{t^2}} \sum_k \frac{k! x^{2k}}{(\ep t)^k (2k)!} \le C_{\ep} e^{-\frac{\ep}{t^2}} \sum_k \frac{x^{2k}}{(\ep t)^kk!} \le C_{\ep} e^{{-\frac{\ep}{t^2}} + \frac{x^2}{\ep t}}.
\end{equation*}

Then for any $R>0$, there exists some $T>0$ such that
Let $P_2 \equiv [-2,2] \times[0, T]$ where $T$ is small enough, we have
\begin{equation}\label{e:Tikhonov estimate}
    \sup_{[-R,R] \times [0,T]} |u_0(x,t)|\le \exp(-\frac{1}{2t^2}).
\end{equation}

We observe that such kind of construction works for general function $\varphi$, with even faster decay rate. Let us set $f_0=t^2$ and $f_{N+1}=e^{-\frac{1}{f_N}}$, $N\geq 0$. From  Lemma \ref{l:fun} we have that 
   \begin{align*}
        | f_N^{(i)}|\leq i!f_N^{1-\ep}(\frac{2}{t}(\frac{8}{\ep})^{N-1}\prod_{j=0}^{N-1}\frac{1}{f_j^\ep})^i  .                                       
    \end{align*}

As in the discussions above, let us denote 
\begin{align*}
    g_N=\frac{2}{t}(\frac{8}{\ep})^{N-1}\prod_{j=0}^{N-1}\frac{1}{f_j^\ep}.
\end{align*}

Then let us consider new Tikhonov type counter-examples with \begin{align*}
    \varphi_N(t)=\left\{\begin{array}{cc}
         f_N(t),&\, t\geq 0 \\
         0,&\, t<0.
    \end{array}\right.
\end{align*}
So we would have that 
\begin{align*}
    |\varphi_N^{(i)}|\leq i!(g_N)^if_N^{1-\ep}.
\end{align*}
Note that for $N=1$, this corresponds to the classical Tikhonov example. We will use this family of examples to construct our solution with prescribed convergence rate of $u$ and expanding rate of the domain.

\subsection{Quadratic nonlinear term}

In this subsection, we will construct the solutions to non-linear backward heat equation, using power series. First we consider the toy model case: 
\begin{enumerate}
    \item the equation is defined on $(-R,R) \times [0,T] \subset \RR \times \RR_+$;
    \item $\vec{\alpha} = \vec{0}$ and $\beta=0$;
    \item $\mathcal{Q}(u,u',u'') = u^2$.
\end{enumerate} 
Hence $u$ is a solution to the following equation in $(-R,R) \times [0,T]$
\begin{align}\label{e:quad}
    \begin{cases}
        (\partial_t + \partial^2_{xx}) u_s &= s \cdot u_s^2, \\
        u_s(x, 0) &= 0.
    \end{cases}
\end{align}

Formally, we expand $u$ in $s,x$ and $t$ to get 
\begin{align*}
    u_s(x,t)=\sum\limits_{j=0}^\infty u^{(j)}(x,t)\frac{s^i}{i!}=\sum\limits_{i,j=0}^\infty u^{(i,j)}(t)\frac{x^i}{i!}\frac{s^j}{j!} \equiv \sum\limits_{i,j=0}^\infty V^{i,j}(t) x^i s^j,
\end{align*}
where for simplicity we define
\begin{equation*}
    V^{i,j}(t) = \frac{u^{(i,j)}(t)}{i!j!}. 
\end{equation*}

We set
\begin{enumerate}
\item For any $j \ge 1$, 
    \begin{equation*}
    u^{(0,j)}(t) = u^{(1,j)} (t) \equiv 0. 
    \end{equation*}
\item For $j = 0$, 
    \begin{align*}
    u^{(2i,0)}(t)=(-1)^i\varphi^{(i)}(t),
    \end{align*}
\end{enumerate}

Recall that 
\[
\varphi(t) = 
\begin{cases}
    e^{-1/t^2}, &t > 0 \\
    0, &t \le0.
\end{cases}
\]
By the estimates \eqref{e:estimate for Tik0} (see \cite{J71}) we have that 
\begin{align}
    |V^{2i,0}|  &= \frac{|u^{(2i,0)}|}{(2i)!}=\frac{|\varphi^{(i)}|}{(2i)!} \leq C_\epsilon \frac{i!}{(2i)!}\frac{1}{(\ep t)^i}e^{-\frac{\epsilon}{t^2}},\\
 |\frac{d^m}{dt^m} V^{2i,0}| &= \frac{|\frac{d^m}{dt^m}u^{(2i,0)}|}{(2i)!}(t)\leq C_\ep \frac{(i+m)!}{(2i)!}\frac{1}{(\ep t)^{i+m}}e^{-\frac{\epsilon}{t^2}},\label{e:derivativecont}
\end{align}
for any $i\geq 0$ and $0<\epsilon<1$.

Plugging the power series to the equation implies
\begin{align*}
    V^{2i+2,j} &= \frac{1}{(2i+2) (2i+1)} \Bigg(  - d_t V^{2i,j} + \sum\limits_{\substack{i_1 +i_2 = i \\ j_1 + j_2=j-1}}  V^{2i_1,j_1} V^{2i_2,j_2} \Bigg).
\end{align*}

Note that
\begin{align*}
    d_t V^{2i,j} &=   \frac{1}{(2i) (2i-1)} d_t \Bigg(  - d_t V^{2i-2,j} + \sum\limits_{\substack{i_1 +i_2 = i-1 \\ j_1 + j_2=j-1}} V^{2i_1,j_1} V^{2i_2,j_2} \Bigg) \\
    &= \frac{1}{(2i) (2i-1)}  \Bigg(  - [d_{t}]^2 V^{2i-2,j} + \sum\limits_{\substack{i_1 +i_2 = i-1 \\ j_1 + j_2=j-1}} \Big(d_t V^{2i_1,j_1} \Big) \cdot V^{2i_2,j_2} +   V^{2i_1,j_1}  \cdot \Big( d_tV^{2i_2,j_2} \Big) \Bigg).
\end{align*}

We can iterate to eliminate the term $[d_t]^{*} V^{*,j}$ using our assumption $V^{0,j} \equiv V^{1,j} \equiv 0$, to obtain the following:
\begin{align}\label{e:u^2 case: expansion of V}
    V^{i,j} =  \sum_{s=0}^{\lfloor i/2 \rfloor -1 } \frac{(i-2-2s)!}{i!}  [-d_t]^{s}\Bigg(\sum\limits_{\substack{ i_1+i_2  = i -2 -2s\\
    j_1 +j_2 =j-1}}   \Big(  V^{i_1,j_1}  \cdot  V^{i_2,j_2} \Big) \Bigg).
\end{align}
It follows directly that for any fixed $m$
\begin{align*}
     [d_t]^m V^{i,j} =  \sum_{s=0}^{\lfloor i/2 \rfloor -1} \frac{(i-2-2s)!}{i!}  [-d_t]^{s+m}\Bigg(\sum\limits_{\substack{ i_1+i_2  = i -2 -2s\\
    j_1 +j_2 =j-1}}   \Big(  V^{i_1,j_1}  \cdot  V^{i_2,j_2} \Big) \Bigg).
\end{align*}

By induction we can easily see that
\begin{enumerate}
    \item $V^{2i+1,j} \equiv 0$ for any $i,j$.
    \item $V^{i,j} \equiv 0$ for any $i < 2j$.
\end{enumerate}

We define
\begin{equation}\label{l:Quad W first decomp}
    W^{2i,j}_m \equiv \frac{(2i)!}{(2i+2m)!} \Big| [d_t]^m\cdot V^{2i.j} \Big|.
\end{equation}

\begin{Lem}
    \begin{align*}
        W^{2i,j}_m \le \sum_{j_1+j_2 = j-1} \sum_{i_1+i_2+m_1+m_2 = i+m-1} W^{2i_1,j_i}_{m_1} \cdot W^{2i_2,j_2}_{m_2}
    \end{align*}
\end{Lem}

\begin{proof}
Fix any $m$, we can compute 
\begin{align*}
    &\quad [d_t]^m \cdot V^{2i,j} \\ &=\pm \sum_{s=0}^{i-1} \frac{(2i-2-2s)!}{(2i)!} [d_t]^{s+m} \Bigg( \sum\limits_{\substack{i_1 + i_2 = i-1-s \\ j_1+j_2=j-1}} (V^{2i_1,j_1} \cdot V^{2i_2,j_2}) \Bigg) \\
    &= \pm \sum_{s=m}^{i+m-1} \frac{(2i+2m-2-2s)!}{(2i)!} [d_t]^{s} \Bigg( \sum\limits_{\substack{i_1 + i_2 = i+m-1-s \\ j_1+j_2=j-1}} (V^{2i_1,j_1} \cdot V^{2i_2,j_2}) \Bigg) \\
    &=\pm \sum_{s=m}^{i+m-1} \frac{(2i+2m-2-2s)!}{(2i)!} \sum_{s_1+s_2 = s} \sum\limits_{\substack{i_1+i_2=i+m-1-s \\ j_1 + j_2 = j-1}} \frac{s!}{s_1! \cdot s_2!}
    \Big([d_t]^{s_1} V^{2i_1,j_1} \Big) \cdot \Big(  [d_t]^{s_2} V^{2i_2,j_2} \Big) \\
    &\le \frac{(2i+2m-2)!}{(2i)!} \sum_{s=m}^{i+m-1} \sum_{s_1+s_2 = s} \sum\limits_{\substack{i_1+i_2=i+m-1-s \\ j_1+ j_2 = j-1}} \Big( \frac{(2i_1)!}{(2i_1+2s_1)!} \Big| [d_t]^{s_1}V^{2i_1,j_1} \Big| \Big) \cdot \Big( \frac{(2i_2)!}{(2i_2+2s_2)!} \Big| [d_t]^{s_2}V^{2i_2,j_2} \Big| \Big)\\
    &\le \frac{(2i+2m)!}{(2i)!} \sum_{j_1+j_2 =j-1} \sum_{i_1+i_2+m_1+m_2=i+m-1} W^{2i_1,j_1}_{m_1} \cdot W^{2i_2,j_2}_{m_2},
\end{align*}
where in the last inequality we use
\begin{align*}
     &\quad  \frac{(2i_1+2s_1)!}{(2i_1)!} \cdot \frac{(2i_2+2s_2)!}{(2i_2)!} \cdot \frac{s!}{s_1! \cdot s_2!} \\
     &\le  \frac{(2i_1+2s_1)!}{(2i_1)!} \cdot \frac{(2i_2+2s_2)!}{(2i_2)!} \cdot \frac{(2s)!}{(2s_1)! \cdot (2 s_2)!} \\
    &= \binom{2i_1 +2s_1}{ 2s_1} \cdot \binom{2i_2 + 2s_2}{2s_2} \cdot (2s)! \\
    &\le \binom{2i_1+2i_2+2s_1+2s_2 }{2s_1 +2s_2} \cdot (2s)! \\
    &= \frac{(2i+2m-2)!}{(2i+2m-2-2s)! \cdot (2s)!} \cdot (2s)! \\
    &= \frac{(2i+2m-2)!}{(2i+2m-2-2s)!}
\end{align*} 
\end{proof}

We now simplify the upper bound as multiplication of $W^{2i_k,0}_{m_k}$. Counting the multiplicity of each term is equivalent to counting the number of ways of decomposition, which in turn corresponds to counting the number of rooted trees with exactly $j$ branching nodes, where the branching nodes have $ k_m $ branches, for $ 1 \leq m \leq j $. From combinatorics, it is well-known that such number is controlled by the Catalan number. For more discussions on Catlan number, see subsection \ref{sec:cat}. This motivates the following lemma which we prove by induction. A similar idea will later be used to handle the general terms, though the technical details there are considerably more delicate.

\begin{Lem}\label{l:Quad case W estimate}
    \begin{align*}
        W^{2i,j}_m \le Cat(j)\sum_{\sum_{k=0}^j i_k + m_k = i+m-j}  \Big(  \prod_{k=0}^j W^{2i_k,0}_{m_k} \Big),
    \end{align*}
    where $Cat(j)$ is the $j$-th Catalan number.  
    \end{Lem}

\begin{proof}
    We have proved the case $j=1$ since the $0$-th Catalan number $Cat(0)=1$. We assume that the result is true for any $1\le j<j_0$. Then by previous Lemma we have
\begin{align*}
        W^{2i,j_0}_m &\le \sum_{j_1+j_2 = j_0-1} \sum_{i_1+i_2+m_1+m_2 = i+m-1} W^{2i_1,j_i}_{m_1} \cdot W^{2i_2,j_2}_{m_2} \\
        &\le \sum_{j_1+j_2 = j_0-1} \sum_{\substack{i_1+i_2+m_1+m_2 \\= i+m-1}}  \Big( Cat(j_1) \sum_{\substack{\sum_{k=0}^{j_1}(i^1_k+m^1_k) \\= i_1+m_1-j_1}} (\prod_{k=0}^{j_1} W^{2i^1_k,0}_{m^1_k}) \Big) \cdot \Big( Cat(j_2) \sum_{\substack{\sum_{k=0}^{j_2}(i^2_k+m^2_k) \\= i_2+m_2-j_2}} (\prod_{k=0}^{j_2} W^{2i^2_k,0}_{m^2_k}) \Big) \\
        &\le \sum_{j_1+j_2 =j_0-1} Cat(j_1)\cdot Cat(j_2) \sum_{\sum_{k=0}^{j_0}i_k + m_k = i-j_0} \prod_{k=0}^{j_0} W^{2i_k,0}_{m_k} \\
        &= Cat(j_0) \sum_{\sum_{k=0}^{j_0}i_k + m_k = i-j_0} \prod_{k=0}^{j_0} W^{2i_k,0}_{m_k},
    \end{align*}
where we use the structure result for Catalan number $\sum_{j_1+j_2=j_0-1} Cat(j_1) \cdot Cat(j_2) = Cat(j_0)$. This finishes the proof by induction.
\end{proof}

Recall that we have the estimate for Catalan number $Cat(j)$
\begin{equation*}
    Cat(j) = \frac{1}{j+1} \binom{2j}{j} \le 4^j.
\end{equation*}
In particular, by Lemma \ref{l:Quad case W estimate} and Lemma \ref{l:prod estimate} we have
    \begin{align*}
        V^{2i,j} = W^{2i,j}_0 &\le Cat(j)\sum_{\sum_k i_k + m_k = i-j}  \Big(  \prod_{k=0}^j \frac{(2i_k)!}{(2i_k+2m_k)!} [d_t]^{m_k} \cdot V^{2i_k,0} \Big)\\
        &\le C_{\ep}^{j+1} \frac{e^{-\frac{(j+1)\ep}{t^2}}}{(\ep t)^{i-j}} \sum_{\sum_k i_k + m_k = i-j}  \Big(  \prod_{k=0}^j\frac{(i_k+m_k)!}{(2i_k+2m_k)!} \Big)\\
        &\le  C_{\ep}^{j+1} \frac{e^{-\frac{(j+1)\ep}{t^2}}}{(\ep t)^{i-j}}  C^{i-j} \frac{(j+1)^{(i-j)}}{(i-j)!} \\
        &\le  C_{\ep}^{j+1}  \cdot \frac{e^{-\frac{(j+1)\ep}{t^2}}}{(\ep t)^{i-j}} \cdot   \frac{(C(j+1))^{(i-j)}}{(i-j)!}
    \end{align*}

Then we can evaluate $u_s(x,t)$ at $s=1$: since $V^{2i,j} \equiv 0 $ for any $i<j$, then
\begin{align*}
    \sum_j \sum_{i\ge j} V^{2i,j} x^{2i} &\le \sum_j  C_{\ep}^{j+1} e^{-\frac{(j+1)\ep}{t^2}} x^{2j} \sum_{i\ge j} \frac{(C(j+1))^{i-j}}{(i-j)!} \frac{ x^{2(i-j)}} {(\ep t)^{i-j} } \\
    &\le \sum_j  C_{\ep}^{j+1} e^{-\frac{(j+1)\ep}{t^2}} x^{2j} \exp{(\frac{Cx^2}{\ep t}\cdot (j+1))} \\
    &\le \exp{( -\frac{\ep}{t^2} + \frac{Cx^2}{\ep t})}\sum_j \Bigg(C^2_{\ep}x^2  \exp{( -\frac{\ep}{t^2} + \frac{Cx^2}{\ep t})} \Bigg)^{j}\\ &=\frac{\exp{( -\frac{\ep}{t^2} + \frac{Cx^2}{\ep t})}}{1-C^2_\ep x^2\exp{( -\frac{\ep}{t^2} + \frac{Cx^2}{\ep t})}}.
\end{align*}
We can choose $|x|^2 \le \frac{c(\ep)}{|t|}$ for some small $c(\ep)>0$ so that $C_\ep^2 \frac{c(\ep)}{|t|} e^{-\frac{\ep}{t^2}+ \frac{Cc(\ep)}{\ep t^2}} < 1/2$ for $|t|\leq 1$. Then this series converges. So we can define \begin{align*}
    u_1(x,t)= \sum_{j} \sum_i V^{2i,j} x^{2i}.
\end{align*}
And as $u_1(0,t)=V^{0,0}=e^{-\frac{1}{t^2}}$, we can see that it is a nontrivial solution of \eqref{e:quad} which also has the desired decay rate.

\subsection{Intermediate case}

Let us then consider the case that 
\begin{align}\label{e:simplecase}
    \mathcal{Q}(u,u',u'')&=\mathcal{Q}(u').
\end{align}

Then in this case, we have 
\begin{align*}
   \frac{dV^{(i,j)}}{dt}+(i+1)(i+2)V^{(i+2,j)}=\sum_{i_2 \ge 2} \frac{\mathcal{Q}^{(0,i_2,0)}}{i_2!}(0,\textbf{0},\textbf{0}) \sum\limits_{\substack{\sum\limits_{}i_{2,k_2}=i\\ \sum\limits_{}j_{2,k_2}=j-1}}\prod_{\substack{k_2=1}}^{i_2} V^{(i_{2,k_2}+1,j_{2,k_2})}(t)(i_{2,k_2}+1).
\end{align*}

We set $V^{0,j} = V^{1,j} \equiv 0$ for any $j \ge 1$ as before. Then as \eqref{e:u^2 case: expansion of V} we have
\begin{equation*}
    V^{i,j} = \sum_{i_2 \ge 2}  \frac{\mathcal{Q}^{(0,i_2,0)}}{i_2!}(0,\textbf{0},\textbf{0}) \sum_{s=1}^{[i/2]} \frac{(i-2s)!}{i!}  [-d_t]^{s-1}\Bigg( \sum\limits_{\substack{\sum\limits_{}i_{2,k_2}=i -2s\\ \sum\limits_{}j_{2,k_2}=j-1}}\prod_{\substack{k_2=1}}^{i_2} V^{(i_{2,k_2}+1,j_{2,k_2})}(t)(i_{2,k_2}+1) \Bigg).
\end{equation*}

We define 
\begin{equation*}
    W^{i,j}_m \equiv \frac{i!}{(i+2m-1)!} [d_t]^m \cdot V^{i,j}.
\end{equation*}

\begin{Lem}\label{l:u' case W estimate}
    \begin{equation*}
        W^{i,j}_m \le \sum_{i_2 \ge 2}  C_Q^{i_2}  \cdot \sum_{\substack{\sum i_{2,k_2} + 2m_{k_2} = i+2m -2\\ \sum j_{2,k_2} = j-1}} \prod_{k_2=1}^{i_2} W^{i_{2,k_2}+1,j_{2,k_2}}_{m_{k_2}}.
    \end{equation*}
\end{Lem}

\begin{proof}
By the analyticity of $\mathcal{Q}$ we have $\frac{\mathcal{Q}^{(0,i_2,0)}}{i_2!}(0,\textbf{0},\textbf{0}) \le C_Q^{i_2}$ for any $i_2$. Hence,
\begin{align*}
    [d_t]^m \cdot V^{i,j} \le  \sum_{i_2 \ge 2}  C_Q^{i_2} \sum_{s=0}^{[\frac{i}{2}]-1} \frac{(i-2-2s)!}{i!}  [d_t]^{s+m}\Bigg( \sum\limits_{\substack{\sum\limits_{}i_{2,k_2}=i -2-2s\\ \sum\limits_{}j_{2,k_2}=j-1}}\prod_{\substack{k_2=1}}^{i_2} V^{(i_{2,k_2}+1,j_{2,k_2})}(t)(i_{2,k_2}+1) \Bigg).
\end{align*}

For each $i_2$ we have
\begin{align*}
    &\quad  \sum_{s=0}^{[\frac{i}{2}]-1} \frac{(i-2-2s)!}{i!}  [d_t]^{s+m}\Bigg( \sum\limits_{\substack{\sum\limits_{}i_{2,k_2}=i -2-2s\\ \sum\limits_{}j_{2,k_2}=j-1}}\prod_{\substack{k_2=1}}^{i_2} V^{(i_{2,k_2}+1,j_{2,k_2})}(t)(i_{2,k_2}+1) \Bigg)\\
        &=  \sum_{s=m}^{[\frac{i}{2}]+m-1} \frac{(i+2m-2-2s)!}{i!}  [d_t]^{s}\Bigg( \sum\limits_{\substack{\sum\limits_{}i_{2,k_2}=i+2m -2-2s\\ \sum\limits_{}j_{2,k_2}=j-1}}\prod_{\substack{k_2=1}}^{i_2} V^{(i_{2,k_2}+1,j_{2,k_2})}(t)(i_{2,k_2}+1) \Bigg) \\
        &=  \sum_{s=m}^{[\frac{i}{2}]+m-1} \frac{(i+2m-2-2s)!}{i!}  \sum_{\sum_{k_2} s_{k_2} = s} \sum\limits_{\substack{\sum\limits_{}i_{2,k_2}=i+2m -2-2s\\ \sum\limits_{}j_{2,k_2}=j-1}}  \frac{s!}{\prod_{k_2} s_{k_2}!} \prod_{\substack{k_2=1}}^{i_2} [d_t]^{s_{k_2}}V^{(i_{2,k_2}+1,j_{2,k_2})}(t)(i_{2,k_2}+1) \\
        &\le \frac{(i+2m-2)!}{i!}   \sum_{s=m}^{[\frac{i}{2}]+m-1} \sum_{\sum_{k_2} s_{k_2} = s} \sum\limits_{\substack{\sum i_{2,k_2}=i+2m -2-2s\\ \sum j_{2,k_2}=j-1}} \prod_{k_2 =1}^{i_2} \Big( \frac{(i_{2,k_2}+1)!}{(i_{2,k_2}+2s_{k_2})!} [d_t]^{s_{k_2}}V^{(i_{2,k_2}+1,j_{2,k_2})}(t) \Big) \\
        &\le \frac{(i+2m-1)!}{i!}  \sum_{\substack{\sum i_{2,k_2} + 2s_{k_2} = i+2m -2\\ \sum j_{2,k_2} = j-1}} \prod_{k_2=1}^{i_2} W^{i_{2,k_2}+1,j_{2,k_2}}_{s_{k_2}}.
    \end{align*}
In the first inequality we use
\begin{align*}
    &\quad \prod_{k_2} \frac{(i_{k_2} +2s_{k_2})! }{i_{k_2}!}\frac{s!}{\prod_{k_2} s_{k_2}!} \\
    &\le \prod_{k_2} \frac{(i_{k_2} +2s_{k_2})! }{i_{k_2}!}\frac{(2s)!}{\prod_{k_2} (2s_{k_2})!}\\
    &= \prod_{k_2} \frac{(i_{k_2}+2s_{k_2})!}{i_{k_2}! \cdot (2s_{k_2})!} \cdot (2s)!\\
    &\le \binom{\sum_{k_2} i_{k_2}+ 2 s_{k_2} }{ \sum_{k_2} 2s_{k_2}} \cdot (2s)! \\
    &= \frac{(i+2m-2)!}{(i+2m-2-2s)!}.
\end{align*}

The proof is now finished.
\end{proof}

Note that every time we break each term $W^{i,j}_s$ into more pieces, the sum on indices $j$ will decrease by 1 and the number of the resulting terms $W^{*,*}_*$ will increase by $\alpha_m-1$ for some $\alpha_m \ge 2$, with the sum of the indices $i$ and $s$ decreasing by $1$. Then we can continue until we lead to the base blocks $W^{*,0}_*$. 

Then the proof will be divided into two parts. First we need to count the multiplicity of each combination of base blocks $W^{*,0}_*$. Then we need to prove the estimate for each combination.

To achieve the first goal, we state the decomposition in the following Lemma, where $l$ is used to denoted $\sum_m \alpha_m$. Hence the number of base blocks is equal to $L=l-j+1$.

\begin{Lem}

    \begin{equation*}
        W^{i,j}_s \le  \sum_{L\ge j+1}  C_Q^{L+j-1}  \cdot Cat(L+j) \cdot  \sum_{\sum_{n} i_n + 2s_n = i+2s-j-1} \prod_{n=1}^{L} {W_{s_n}^{i_n+1,0}}
    \end{equation*}
\end{Lem}

\begin{proof}
    We will prove the Lemma by induction. The base case $j=1$ is finished by the previous Lemma \ref{l:u' case W estimate}. Then we assume the result is true for any $j<j_0$. Therefore, by Lemma \ref{l:u' case W estimate}
    \begin{align*}
        W^{i,j_0}_s &\le \sum_{\alpha \ge 2}  C_Q^{\alpha}  \cdot \sum_{\substack{\sum i_{k} + 2s_{k} = i+2s -2\\ \sum j_{k} = j_0-1}} \prod_{k=1}^{\alpha} W^{i_{k}+1,j_{k}}_{s_{k}} \\
        &\le \sum_{\alpha \ge 2}  C_Q^{\alpha}  \cdot \sum_{\substack{\sum i_{k} + 2s_{k} = i+2s -2\\ \sum j_{k} = j_0-1}} \prod_{k=1}^{\alpha} \Bigg( \sum_{L_k \ge j_k+1}  C_Q^{L_k+j_k-1} \cdot Cat(L_k+j_k)  \cdot \sum_{\sum_{n} i_{k,n} + 2s_{k,n} = i_k+2s_k-j_k} \prod_{n=1}^{L_k} {W_{s_{k,n}}^{i_{k,n}+1,0}} \Bigg)
    \end{align*}

Set $L \equiv \sum_{k=1}^{\alpha} L_k$. Since each $L_k \ge j_k+1 \ge 1$, we have $2\le\alpha \le L$ and $L \ge \sum_{k=1}^{\alpha} (j_k+1) \ge j_0 -1 +\alpha \ge j_0 +1$. Hence, if we sum on $L$ then
\begin{align*}
    W^{i,j_0}_s &\le \sum_{L \ge j_0+1} C_Q^{L +j_0 -1} \sum_{\alpha \ge 2}^L \Big( \sum_{\substack{\sum_k L_k = L \\ \sum_k j_k = j_0 - 1}} \prod_{k=1}^{\alpha} Cat(L_k+j
    _k) \Big) \cdot \Big(\sum_{\sum_n i_n +2 s_n = i+2s-j_0-1}\prod_{n=1}^{L} W^{i_n+1,0}_{s_n} \Big) \\
    &\le \sum_{L \ge j_0+1} C_Q^{L +j_0 -1} Cat(L+J) \cdot \Big(\sum_{\sum_n i_n +2 s_n = i+2s-j_0-1}\prod_{n=1}^{L} W^{i_n+1,0}_{s_n} \Big),
\end{align*}
where in the last inequality we use the recursive formula for Catalan numbers, see Lemma \ref{l:Catalan variable branching}:
\begin{align*}
    \sum_{\alpha \ge 2}^L \Big( \sum_{\substack{\sum_k L_k = L \\ \sum_k j_k = j_0 - 1}} \prod_{k=1}^{\alpha} Cat(L_k+j
    _k) \Big) \le \sum_{\alpha \ge 2}^{L+j_0-1} \Big( \sum_{\sum_k (L_k+j_k) = L +j_0-1 } \prod_{k=1}^{\alpha} Cat(L_k+j
    _k) \Big) \le Cat(L+j_0).
\end{align*}
\end{proof}

Then let we estimate each $W^{i+1,0}_s$. Since $V^{i+1,0}=0$ if $i+1$ is odd, we can assume $i+1$ is even.
\begin{align*}
    W^{i+1,0}_s &= \frac{(i+1)!}{(i+2s)!} [d_t]^sV^{i+1,0} \\
    &\le \frac{(i+1)!}{(i+2s)!}  \cdot \Big( \frac{(\frac{i+1}{2}+s)!}{(i+1)!} C_{\epsilon} \frac{1}{(\ep t)^{\frac{i+1}{2}+s}} e^{-\frac{\epsilon}{t^2}} \Big) \\
    &\le \frac{C_{\ep} e^{-\frac{\epsilon}{t^2}} }{(\ep t)^{\frac{i+1}{2}+s}} \cdot \frac{(\frac{i+1}{2}+s)!}{(i+2s)!}.
\end{align*}

Then by Lemma \ref{l:prod estimate}, noting that $\sum_{n} \frac{i_n+1}{2} + s_n = \frac{i-j-1 + (l-j+1)}{2} = \frac{i+l-2j}{2}$,
\begin{align*}
    \prod_{n=1}^{l-j+1}  \frac{(\frac{i_n+1}{2}+s_n)!}{(i_n+2s_n)!} \le C^{i+l-2j} \cdot \frac{(l-j+1)^{\frac{i+l-2j}{2}}}{[\frac{i+l-2j}{2}]!}.
\end{align*}
and thus
\begin{align*}
    \sum_{\sum_n i_n +2 s_n = i-j-1}\prod_{n=1}^{l-j+1} W^{i_n+1,0}_{s_n} &\le \sum_{\sum_n i_n +2 s_n = i-j-1}\prod_{n=1}^{l-j+1} \Big( \frac{C_{\ep} e^{-\frac{\epsilon}{t^2}} }{(\ep t)^{\frac{i_n+1}{2}+s_n}} \cdot \frac{(\frac{i_n+1}{2}+s_n)!}{(i_n+2s_n)!} \Big)
    \\
    &\le \frac{C_{\ep}^{l-j+1} e^{-\frac{\ep}{t^2} \cdot (l-j+1)}}{(\ep t)^{\frac{i+l-2j}{2}}} \cdot \sum_{\sum_n i_n +2 s_n = i-j-1}\prod_{n=1}^{l-j+1}  \frac{(\frac{i_n+1}{2}+s_n)!}{(i_n+2s_n)!} \\
\end{align*}

Since $l \ge 2j$, we can define $k=l-2j \ge 0$ to simplify the series

\begin{align*}
    \sum_{j} \sum_i V^{i,j} x^i &= \sum_{j} \sum_i \frac{x^i}{i}W^{i,j}_0 \\
    &\le \sum_{j} \sum_i \sum_{l \ge 2j} C^l_Q \cdot Cat(l+1) \cdot Cat(j) \cdot |x|^i \cdot  \sum_{\sum_n i_n +2 s_n = i-j-1}\prod_{n=1}^{l-j+1} W^{i_n+1,0}_{s_n} \\
    &\le \sum_{j} C^j \sum_{k\ge 0} C^{k} e^{-\frac{\ep}{t^2}(j+k+1)} \sum_i  \frac{(C|x|)^{i}}{(\ep t)^{\frac{i+k}{2}}} \cdot \frac{(j+k+1)^{\frac{i+k}{2}}}{[\frac{i+k}{2}]!}.
\end{align*}

We simply consider $|x|\ge 1$ so that $(C|x|)^i \le (C|x|)^{i+k}$. Otherwise we just use the trivial bound $(C|x|)^i \le C^{i+k}$. Therefore we have
\begin{align*}
    \sum_{j} \sum_i V^{i,j} x^i &\le \sum_{j} C^j \sum_{k\ge 0} C^{k} e^{-\frac{\ep}{t^2}(j+k+1)} e^{\frac{C(j+k+1)|x|^2}{\ep t}} \\
    &\le   C e^{-\frac{\ep}{t^2}+ \frac{C|x|^2}{\ep t} }\sum_j \Big( C e^{-\frac{\ep}{t^2}+ \frac{C|x|^2}{\ep t} }\Big)^j \cdot \sum_k \Big( C e^{-\frac{\ep}{t^2}+ \frac{C|x|^2}{\ep t} }\Big)^k
\end{align*}

We can choose $|x|^2 \le \frac{c}{|t|}$ for some small $c(Q)>0$ so that $C e^{-\frac{\ep}{t^2}+ \frac{C|x|^2}{\ep t}} < 1/2$. Then this series converges. So we can define \begin{align*}
    u_1(x,t)= \sum_{j} \sum_i V^{i,j} x^i.
\end{align*}
And as $u_1(0,t)=V^{0,0}=e^{-\frac{1}{t^2}}$, we can see that it is a nontrivial solution of \eqref{e:simplecase}.

\section{General construction}

In this section, we prove the Theorem \ref{t:Tikhonov}, where we construct the solutions to the following general equation
\begin{align}
    \partial_t u_s + u_s'' = s \cdot t^{\theta_0} \mathcal{Q}(t^{\theta_1} u, t^{{\theta_2}}u',t^{\theta_3} u''),\label{e:perturbequ}
\end{align}
where $\mathcal{Q}$ is analytic around $(0,0,0)$ with $\mathcal{Q}(0,0,0) = |\nabla \mathcal{Q}(0,0,0)| = 0$. In particular, we have $|\frac{\mathcal{Q}^{(i_0,i_1,i_2)}}{i_0!i_1!i_2!}(0,0,0) |\le C_Q^{i_0+i_1+i_2}$ for some constant $C_Q$.  In the first subsection, we construct the solutions out of the standard Tikhonov example with decay rate $e^{-\frac{1}{t^2}}$. We will generalize this in the second subsection, where we obtain the solutions with arbitrary decay rate.

\subsection{Special region}

Similarly as before, we set the following.
\begin{align*}
    u_s(x,t)=\sum\limits_{j=0}^\infty u^{(j)}(x,t)\frac{s^i}{i!}=\sum\limits_{i,j=0}^\infty u^{(i,j)}(t)\frac{x^i}{i!}\frac{s^j}{j!},
\end{align*}
with $V^{i,j}(t) = \frac{u^{(i,j)}}{i!j!}(t)$. Moreover, $V^{0,j} = V^{1,j} = 0$ for any $j \ge 1$. For $j = 0$, we set
    \begin{align*}
    u^{(2i,0)}(t)=(-1)^i\varphi^{(i)}(t),
    \end{align*}
with
\begin{equation*}
    \varphi(t) = 
\begin{cases}
    e^{-1/t^2}, &t > 0 \\
    0, &t \le0.
\end{cases}
\end{equation*}

By comparing the coefficients, we have the following equation 
\begin{align*}
    &\frac{d  V^{(i,j)}}{dt} + (i+1)(i+2)  V^{(i+2,j)}=  \sum_{i_0,i_1,i_2\geq 0}\frac{\mathcal{Q}^{(i_0,i_1,i_2)}}{i_0!i_1!i_2!}(0,\textbf{0},\textbf{0}) \cdot t^{\theta_0 + \theta_1 i_0 + \theta_2 i_1 +\theta_3 i_2}\\ &  \sum_{\substack{\sum_{m}\sum_{k_m} i_{m,k_m}=i\\ \sum_{m}\sum_{k_m} j_{m,k_m}=j-1}} \Big( \prod_{k_0=1}^{i_0}  V^{(i_{0,k_0},j_{0,k_0})} \Big) \cdot \Big( \prod_{k_1=1}^{i_1}   V^{(i_{1,k_1}+1,j_{1,k_1})}  (i_{1,k_1}+1)  \Big) \cdot \Big( \prod_{k_2=1}^{i_2} V^{(i_{2,k_2}+2,j_{2,k_2})}  (i_{2,k_2}+2)(i_{2,k_2}+1) \Big).
\end{align*}

Then similarly as in the simple case, we can write $V^{i,j}$ as follows
\begin{align*}
    V^{i,j} &= \sum_{i_0,i_1,i_2\geq 0}\frac{\mathcal{Q}^{(i_0,i_1,i_2)}}{i_0!i_1!i_2!}(0,\textbf{0},\textbf{0}) \cdot \sum_{s=0}^{[\frac{i}{2}]-1} \frac{(i-2-2s)!}{i!} [-d_t]^{s} 
    \Bigg( t^{-\Theta} \cdot \sum_{\substack{\sum_{m}\sum_{k_m} i_{m,k_m}=i -2 -2s\\ \sum_{m}\sum_{k_m} j_{m,k_m}=j-1}} \Big(  \prod_{k_0=1}^{i_0}  V^{(i_{0,k_0},j_{0,k_0})} \Big) \\
    &\Big(  \prod_{k_1=1}^{i_1}   V^{(i_{1,k_1}+1,j_{1,k_1})}  (i_{1,k_1}+1)  \Big) \cdot \Big( \prod_{k_2=1}^{i_2} V^{(i_{2,k_2}+2,j_{2,k_2})}  (i_{2,k_2}+2)(i_{2,k_2}+1) \Big) \Bigg).
\end{align*}
For simplicity we write $\Theta \equiv - (\theta_0 + \theta_1 i_0 + \theta_2 i_1 + \theta_3 i_2)$.

In the following we consider the small time interval, say $0<t\le 1$. 
Since we want to find an upper bound for $|V^{i,j}|$, we could simplify the notation by considering the case $\theta_i<0$ and thus $\Theta > 0$. Since $i_0+i_1+i_2 \ge 2$, we can also choose some large $\theta>0$ such that $\Theta \le \theta(i_0+i_1+i_2).$

We define 
\begin{equation*}
    W^{i,j}_s \equiv \frac{i!}{(i+2s-2)!} \Big| [d_t]^s V^{i,j} \Big|.
\end{equation*}

\begin{Lem}\label{l:general case W estimate}
    \begin{equation*}
         W^{i,j}_n \le \sum_{\substack{i_0,i_1,i_2\geq 0\\ i_0+i_1+i_2 \ge 2}} (\frac{C}{t})^{\theta(i_0+i_1+i_2)} \cdot \sum_{\substack{2s_0 + \sum_{m}\sum_{k_m} i_{m,k_m} + 2s_{m,k+m} =i +2n -2\\ \sum_{m}\sum_{k_m} j_{m,k_m}=j-1}} \Big( \frac{s_0!\cdot C^{s_0}}{(2s_0)!\cdot t^{s_0}} \cdot \prod_{m,k_m} W^{i_{m,k_m}+m, j_{m,k_m}}_{s_{m,k_m}} \Big).
    \end{equation*}
\end{Lem}

\begin{proof}
    We follow the proof of Lemma \ref{l:Quad W first decomp}. 

    We use $|\frac{Q^{(i_0,i_1,i_2)}}{i_0!i_1!i_2!}(0,\textbf{0},\textbf{0})| \le C^{i_0+i_1+i_2}$ to obtain
\begin{align*}
    &[d_t]^n\cdot V^{i,j} \le \sum_{i_0,i_1,i_2\geq 0} C^{i_0+i_1+i_2} \cdot \sum_{s=0}^{[\frac{i}{2}]-1} \frac{(i-2-2s)!}{i!} \Bigg| [-d_t]^{s+n} 
    \Bigg( t^{-\Theta} \cdot \sum_{\substack{\sum_{m}\sum_{k_m} i_{m,k_m}=i -2 -2s\\ \sum_{m}\sum_{k_m} j_{m,k_m}=j-1}} \Big( \prod_{k_0=1}^{i_0}  V^{(i_{0,k_0},j_{0,k_0})} \Big) \\
    &\quad \quad \Big( \prod_{k_1=1}^{i_1}   V^{(i_{1,k_1}+1,j_{1,k_1})}  (i_{1,k_1}+1)  \Big) \cdot \Big( \prod_{k_2=1}^{i_2} V^{(i_{2,k_2}+2,j_{2,k_2})}  (i_{2,k_2}+2)(i_{2,k_2}+1) \Big) \Bigg) \Bigg| \\
    &= \sum_{i_0,i_1,i_2\geq 0} C_Q^{i_0+i_1+i_2} \cdot \sum_{s=n}^{[\frac{i}{2}]+n-1}  \sum_{\substack{\sum_{m}\sum_{k_m} i_{m,k_m}=i +2n -2 -2s\\ \sum_{m}\sum_{k_m} j_{m,k_m}=j-1 \\ s_0 + \sum_m \sum_{k_m} s_{m,k_m} = s}}    \frac{(i+2n-2-2s)!}{i!} \cdot  \frac{s!}{s_0! \prod s_{m,k_m}!}  \Bigg|  \Big( [d_t]^{s_0} \cdot t^{-\Theta} \Big) \\
    & \Big(  \prod_{k_0=1}^{i_0} [d_s]^{s_{0,k_0}}  V^{(i_{0,k_0},j_{0,k_0})} \Big)  \Big( \prod_{k_1=1}^{i_1} [d_s]^{s_{1,k_1}}  V^{(i_{1,k_1}+1,j_{1,k_1})}  (i_{1,k_1}+1)  \Big) \Big( \prod_{k_2=1}^{i_2} [d_s]^{s_{2,k_2}}V^{(i_{2,k_2}+2,j_{2,k_2})}  (i_{2,k_2}+2)(i_{2,k_2}+1) \Big) \Bigg|. 
\end{align*}

We claim that:
\begin{align*}
    \frac{(i+2n-2-2s)!}{i!}  \cdot \frac{s!}{s_0! \prod s_{m,k_m}!}  \le \frac{(i+2n-2)!}{i!} \cdot \frac{1}{(2s_0)!} \cdot \prod \frac{i_{m,k_m}!}{(i_{m,k_m}+2s_{m,k_m})!}.
\end{align*}

With this claim we can easily prove that
\begin{equation*}
    W^{i,j}_n \le \sum_{\substack{i_0,i_1,i_2\geq 0\\ i_0+i_1+i_2 \ge 2}} C_Q^{i_0+i_1+i_2} \cdot \sum_{\substack{2s_0 + \sum_{m}\sum_{k_m} i_{m,k_m} + 2s_{m,k+m} =i +2n -2\\ \sum_{m}\sum_{k_m} j_{m,k_m}=j-1}} \Big( \frac{[d_t]^{s_0} \cdot t^{-\Theta}}{(2s_0)!} \cdot \prod_{m,k_m} W^{i_{m,k_m}+m, j_{m,k_m}}_{s_{m,k_m}} \Big).
\end{equation*}

Then we use the following inequality: Since $\Theta > 0$
$$ 
\frac{[d_t]^{s_0} \cdot t^{-\Theta}}{(2s_0)!} \le \frac{(\Theta+s_0-1)! }{(\Theta-1)!(2s_0)!t^{\Theta+s_0}} \le \frac{s_0!}{(2s_0)!} (\frac{C}{t})^{\Theta+s_0} .
$$

The the proof will be finished if we prove the claim as follows:
\begin{align*}
    &\quad \prod \frac{(i_{m,k_m}+2s_{m,k_m})!}{i_{m,k_m}!} \cdot \frac{s!}{s_0! \prod s_{m,k_m}!} \\
    &\le \prod \frac{(i_{m,k_m}+2s_{m,k_m})!}{i_{m,k_m}!} \cdot \frac{(2s)!}{(2s_0)! \prod (2s_{m,k_m})!} \\
    &= \prod \binom{i_{m,k_m}+2s_{m,k_m}}{ 2s_{m,k_m}}  \cdot \frac{(2s)!}{(2s_0)!} \\
    &\le \binom{\sum i_{m,k_m}+2s_{m,k_m} }{ \sum 2s_{m,k_m}} \cdot \frac{(2s)!}{(2s_0)!} \\
    &\le \binom{i+2n-2} {2s} \cdot \frac{(2s)!}{(2s_0)!} \\
    &= \frac{(i+2n-2)!}{(i+2n-2-2s)! (2s_0)!}.
\end{align*}
\end{proof}

Then we can decompose $W^{i,j}_n$ into the base blocks $W^{*,0}_*$.

\begin{Lem}\label{l:example general decomp}
For any $v=0,1,2$, we have
\begin{align*}
     W^{I+v,J}_N &\le \sum_{L \ge 2J} (\frac{C}{t})^{\theta\cdot L}   \sum_{\substack{ L_0+L_1+L_2 = L \\ J_0 +J_1 +J_2 = J-1  \\ \sum_n 2s_n + \sum_{m}\sum_{k_m} i_{m,k_m} +  2s_{m,k_m} \\=I +2N +v -2 -2J_0-J_1}} T_{L,J,p} \cdot \Big( \prod_{n=1}^{J} \frac{s_n!C^{s_n}}{(2s_n)!t^{s_n}} \cdot \prod_{m=0}^2 \prod_{k_m}^{L_m-J_m} W^{i_{m,k_m}+m, 0}_{s_{m,k_m}} \Big),
\end{align*}
where $T_{L,J;p} = Cat(L+1) \cdot Cat(J) \cdot \frac{L!}{L_0!L_1!L_2!} \cdot \frac{(J-1)!}{J_0!J_1!J_2!}$ is a constant.
\end{Lem}

\begin{proof}
    
We prove the lemma by induction. The case $j=1$ is proved by the lemma \ref{l:general case W estimate}. Now we assume that the result holds for any $j < J$.

Then by Lemma \ref{l:general case W estimate} we have 
\begin{align*}
    W^{I+v,J}_N &\le \sum_{l_0 \geq 2} (\frac{C_Q}{t})^{l_0} \cdot \sum_{\substack{l_{0;0}+l_{0;1}+l_{0;2} = l_0 \\ \sum_{m}\sum_{k_m} j_{m,k_m}=J-1 \\ 2s_{0} + \sum_{m}\sum_{k_m} i_{m,k_m} + 2s_{m,k_m} =I +2N +v -2
    }} \Big( \frac{s_{0}!C^{s_0}}{(2s_{0;0})!t^{s_{0}}} \cdot \prod_{m=0}^2 \prod_{k_m}^{l_{0;m}} W^{i_{m,k_m}+m, j_{m,k_m}}_{s_{m,k_m}} \Big).
\end{align*}

Then we keep those terms $W^{*,0}_*$ and apply the induction hypothesis to each $W^{*,j_{m,k_m}}_*$ with $j_{m,k_m} \ge 1$.  Suppose we have $j_0$ number of term with $j_{m,k_m} \ge 1$. Let $j_{0,2}$ be the number of $W^{i_{m,k_m}+2,j_{m,k_m}}_{s_{m,k_m}}$ with $j_{m,k_m} \ge 1$. Then $j_{0,2} \le l_{0,2}$. Same for $j_{0,1}$ and $j_{m,k_m;1}$. Hence $j_0 = j_{0,1} + j_{0,2} + j_{0,3} \le l_0$ and then we obtain
\begin{align*}
    & W^{i_{m,k_m}+m, j_{m,k_m}}_{s_{m,k_m}} \le \sum_{l_{m,k_m} \geq 2} (\frac{C_Q}{t})^{\Theta \cdot l_{m,k_m}} \sum_{\substack{ l_{m,k_m;0}+l_{m,k_m;1}+l_{m,k_m;2} = l_{m,k_m} \\ j_{m,k_m;0} +j_{m,k_m;1} +j_{m,k_m;2} = j_{m,k_m} -1}}   T^{m,k_m}_{l,j;p} \\
    & \sum_{\substack{ \sum_n 2s_{m,k_m;n} + \sum_{p}\sum_{k_p} i_{m,k_m;p,k_p} + 2s_{m,k_m;p,k_p} \\
    =i_{m,k_m} +2s_{m,k_m} +m-2 -2j_{m,k_m;0}-j_{m,k_m;1} }} \Big( \prod_{n=1}^{j_{m,k_m}} \frac{s_{m,k_m;n}!C^{s_{m,k_m;n}}}{(2s_{m,k_m;n})!t^{s_{m,k_m;n}}} \cdot \prod_{p=0}^2 \prod_{k_p}^{l_{m,k_m;p} - j_{m,k_m;p}} W^{i_{m,k_m;p,k_p}+p, 0}_{s_{m,k_m;p,k_p}} \Big) 
\end{align*}
where 
\begin{align*}
    T^{m,k_m}_{l,j;p} = Cat(l_{m,k_m}+1) \cdot Cat(j_{m,k_m}) \cdot \frac{l_{m,k_m}!}{l_{m,k_m;0}!l_{m,k_m;1}! l_{m,k_m;2}!} \cdot  \frac{(j_{m,k_m}-1)!}{j_{m,k_m;0}!j_{m,k_m;1}! j_{m,k_m;2}!}.
\end{align*}

Note that $m=0,1,2$ and $k_m$ range from 1 to $j_{0,m}$. 

Note $l_{m,k_m} = \sum_{p=0}^2 l_{m,k_m;p}$ and set $L= l_0 + \sum_{m,k_m}l_{m,k_m} = \sum_p (l_{0,p} + \sum_{m,k_m} l_{m,k_m;p})$. Moreover, we set $L_p = l_{0,p} + \sum_{m,k_m} l_{m,k_m;p}$. Then $L = L_0 + L_1+L_2$. 

Note that $j_{m,k_m} -1 = \sum_{p=0}^2 j_{m,k_m;p}$ and then $J -1 =  \sum_{m,k_m} j_{m,k_m} = j_0 + \sum_p \sum_{m,k_m} j_{m,k_m;p}$. Moreover, we set $J_p = j_{0,p}+  \sum_{m,k_m} j_{m,k_m;p}$. Then $J-1 = J_0 + J_1+J_2$.

Then we have 
\begin{align*}
    &\quad 2s_0 + \sum_n \sum_{m} \sum_{k_m} 2s_{m,k_m;n} + \sum_{m}\sum_{k_m} \sum_{p}\sum_{k_p} (i_{m,k_m;p,k_p} + 2s_{m,k_m;p,k_p}) \\
    &= 2s_0 + \sum_m \sum_{k_m} (i_{m,k_m}  + 2s_{m,k_m} +m-2 - 2j_{m,k_m;0} - j_{m,k_m;1} ) \\
    &= I + 2N +v -2 - 2j_{0,0} - j_{0,1} -2 \sum_m \sum_{m,k_m} (2j_{m,k_m;0} - j_{m,k_m;1})\\
    &= I +2N+v -2 -2J_0 - J_1 .
\end{align*}

Then we rearrange the terms to get
\begin{align*}
    &W^{I+v,J}_N \le \sum_{L \ge 2J} (\frac{C}{t})^L   \sum_{\substack{ L_0+ L_1 +L_2 =L \\ J_0+J_1+J_2 = J-1  \\ s_0 + \sum_n 2s_{n} + \sum_{p}\sum_{k_p} i_{p,k_p} + 2s_{p,k_p} \\
    =I + 2N +v - 2-2J_0 - J_1 }}   \frac{s_0!C^{s_0}}{(2s_0)! \cdot t^{s_0}} \cdot \prod_{n=1}^{J-1} \frac{s_n!C^{s_n}}{(2s_n)! \cdot t^{s_n}} \cdot \prod_{p=0}^2 \prod_{k_p}^{L_p - J_p}  W^{i_{p,k_p}+p,0}_{s_{p,k_p}} \\
    &  \sum_{\substack{l_0 \ge 2 \\ j_0 \ge 1}} \sum_{\substack{\sum_p l_{0,p} = l_0  \\ \sum_p j_{0,p} =j_0}} \sum_{\substack{ \sum_{m=0}^2\sum_{k_m=0}^{j_{0,m}} l_{m,k_m} = L - l_0 \\ \sum_{m=0}^2\sum_{k_m=0}^{j_{0,m}} j_{m,k_m} = J-1 }} \sum_{\substack{ \sum_{p=0}^2  l_{m,k_m;p} = l_{m,k_m} \\ \sum_{p=0}^2  j_{m,k_m;p} = j_{m,k_m}-1 \\ \sum_{m=0}^2 \sum_{k_m=0}^{j_{0,m}} l_{m,k_m;p} = L_p - l_{0,p} \\ \sum_{m=0}^2 \sum_{k_m=0}^{j_{0,m}} j_{m,k_m;p} = J_p - j_{0,p} \\  p=0,1,2}} \prod_{m=0}^2 \prod_{k_m=1}^{j_{0,m}}  T^{m,k_m}_{l,j} 
\end{align*}

Recall that 
\begin{align*}
    T^{m,k_m}_{l,j;p} = Cat(l_{m,k_m}+1) \cdot Cat(j_{m,k_m})\cdot\frac{l_{m,k_m}!}{l_{m,k_m;0}!l_{m,k_m;1}! l_{m,k_m;2}!} \cdot  \frac{(j_{m,k_m}-1)!}{j_{m,k_m;0}!j_{m,k_m;1}! j_{m,k_m;2}!}.
\end{align*}

We {claim} that we can combine Lemma \ref{l:sumprod} and Lemma \ref{l:Catalan variable branching} to conclude the following, which would finish the proof, 
\begin{align*}
    &\quad\sum_{\substack{l_0 \ge 2 \\ j_0 \ge 1}} \sum_{\substack{\sum_p l_{0,p} = l_0  \\ \sum_p j_{0,p} =j_0}} \sum_{\substack{ \sum_{m=0}^2\sum_{k_m=0}^{j_{0,m}} l_{m,k_m} = L - l_0 \\ \sum_{m=0}^2\sum_{k_m=0}^{j_{0,m}} j_{m,k_m} = J-1  }} \sum_{\substack{ \sum_{p=0}^2  l_{m,k_m;p} = l_{m,k_m} \\ \sum_{p=0}^2  j_{m,k_m;p} = j_{m,k_m}-1 \\ \sum_{m=0}^2 \sum_{k_m=0}^{j_{0,m}} l_{m,k_m;p} = L_p - l_{0,p} \\ \sum_{m=0}^2 \sum_{k_m=0}^{j_{0,m}} j_{m,k_m;p} = J_p - j_{0,p} \\  p=0,1,2}} \prod_{m=0}^2 \prod_{k_m=1}^{j_{0,m}}  T^{m,k_m}_{l,j}  \\
    = &\Big[ \sum_{l_0 \ge 2} \sum_{\sum_{m,k_m} l_{m,k_m} = L -l_0} \prod_{m,k_m} Cat(l_{m,k_m}+1) \Big] \cdot \Big[ \sum_{j_0 \ge 1} \sum_{\sum_{m,k_m} j_{m,k_m}=J-1}\prod_{m,k_m} Cat(j_{m,k_m}) \Big] \cdot           \\
    & \Big[\sum_{\sum_p l_{0,p} = l_0}  \sum_{\substack{ \sum_{m,k_m} l_{m,k_m;p} = L_p - l_p\\ p =0,1,2;\\ \sum_p l_{m,k_m;p} = l_{m,k_m}}}  \prod_{m,k_m} \frac{l_{m,k_m}!}{l_{m,k_m;0}!l_{m,k_m;1}!l_{m,k_m;2}!} \Big]\cdot  \\ 
    &\Big[ \sum_{\sum_p j_{0,p} = j_0}  \sum_{\substack{ \sum_{m,k_m} j_{m,k_m;p} = J_p - j_p\\ p =0,1,2;\\ \sum_p j_{m,k_m;p} = j_{m,k_m}-1}}  \prod_{m,k_m} \frac{(j_{m,k_m}-1)!}{j_{m,k_m;0}!j_{m,k_m;1}!j_{m,k_m;2}!} \Big]   \\
    &\le Cat(L+1) \cdot Cat(J) \cdot \frac{L!}{L_0!L_1!L_2!} \cdot \frac{(J-1)!}{J_0!J_1!J_2!}. \\
\end{align*}.

To prove this claim, we firstly note that by Lemma \ref{l:Catalan variable branching}, since $j_0 \le l_0$, we have
\begin{align*}
    \sum_{l_0 \ge 2} \sum_{\sum_{m,k_m} l_{m,k_m} = L -l_0} \prod_{m,k_m} Cat(l_{m,k_m}+1) = \sum_{l_0 \ge 2} \sum_{\sum_{m,k_m} (l_{m,k_m}+1) = L +j_0-l_0} \prod_{m,k_m} Cat(l_{m,k_m}+1) \le Cat(L+1).
\end{align*}
Also by Lemma \ref{l:Catalan variable branching} we have
\begin{align*}
    \sum_{j_0 \ge 1} \sum_{\sum_{m,k_m} j_{m,k_m}=J-1}\prod_{m,k_m} Cat(j_{m,k_m}) = Cat(J).
\end{align*}

Since $\sum_{m,k_m} l_{m,k_m} = L- l_0$, then by \eqref{e:counting row and column} we have
\begin{align*}
    & \quad \sum_{\sum_p l_{0,p} = l_0}  \sum_{\substack{ \sum_{m,k_m} l_{m,k_m;p} = L_p - l_p\\ p =0,1,2;\\ \sum_p l_{m,k_m;p} = l_{m,k_m}}}  \prod_{m,k_m} \frac{l_{m,k_m}!}{l_{m,k_m;0}!l_{m,k_m;1}!l_{m,k_m;2}!} \\&= \sum_{\sum_p l_{0,p} = l_0} \frac{(L-l_0)!}{(L_0-l_{0,p})!(L_1-L_{0,1})!(L_2-l_{0,2})!} \le \frac{L!}{L_0!L_1!L_2!}. 
\end{align*}

Also since $\sum_{m,k_m} (j_{m,k_m}-1) = J-1-j_0$, then by \eqref{e:counting row and column} we have
\begin{align*}
    & \quad\sum_{\sum_p j_{0,p} = j_0}  \sum_{\substack{ \sum_{m,k_m} j_{m,k_m;p} = J_p - j_p\\ p =0,1,2;\\ \sum_p j_{m,k_m;p} = j_{m,k_m}-1}}  \prod_{m,k_m} \frac{(j_{m,k_m}-1)!}{j_{m,k_m;0}!j_{m,k_m;1}!j_{m,k_m;2}!} \\
    &= \sum_{\sum_p j_{0,p} = j_0} \frac{(J-1-j_0)!}{(J_0 - j_{0,p})!(J_1 - j_{1,p})!(J_2 - j_{2,p})!} \le \frac{(J-1)!}{(J_0!J_1!J_2!}.
\end{align*}

Thus the claim is true, and this finishes the proof by induction. 

\end{proof}

Finally we can estimate each term $W^{*,*}_*$. For $m=0,1,2$, by \eqref{e:derivativecont},
\begin{align*}
    W^{i+m,0}_s &= \frac{(i+m)!}{(i+m+2s-2)!} \Big| [d_t]^{s} V^{i+m,0} \Big| \\
    &\le \frac{[\frac{i+m}{2}+s]!}{(i+m+2s-2)!} \cdot \frac{C_{\ep}e^{-\frac{\ep}{t^2}}}{(\ep t)^{\frac{i+m}{2}+s}}.
\end{align*}

Then we can apply Lemma \ref{l:example general decomp} to $W^{I,J}_0$ to get
\begin{align*}
    W^{I,J}_0 &\le \sum_{L \ge 2J} (\frac{C}{t})^{\theta\cdot L}   \sum_{\substack{ L_0+L_1+L_2 = L \\ J_0 +J_1 +J_2 = J-1  \\ \sum_n 2s_n + \sum_{m}\sum_{k_m} i_{m,k_m} +  2s_{m,k_m} \\=I  -2 -2J_0-J_1}} T_{L,J,p} \cdot \Big( \prod_{n=1}^{J} \frac{s_n!C^{s_n}}{(2s_n)!t^{s_n}} \cdot \prod_{m=0}^2 \prod_{k_m}^{L_m-J_m} W^{i_{m,k_m}+m, 0}_{s_{m,k_m}} \Big).
\end{align*}

Note that $\sum_n 2s_n + \sum_{m,k_m} (i_{m,k_m} +m + 2s_{m,k_m}) = I + L_1 +2L_2 - 2J = I + (L-2J) - 2L_0 -L_1$. For simplicity we let $\M = [\frac{I + (L-2J) - 2L_0 -L_1}{2}]$. Then by Lemma \ref{l:prod estimate} we have
\begin{align*}
    &\quad \quad \prod_{n=1}^J \frac{s_n! C^{s_n}}{(2s_n)! t^{s_n}}\prod_m \prod_{k_m} W^{i_{m,k_m} +m,0}_{s_{m,k_m}} \\
    &\le \prod_{n=1}^J \frac{s_n! C^{s_n}}{(2s_n)! t^{s_n}}\prod_m \prod_{k_m} \Big( \frac{[\frac{i_{m,k_m}+m}{2}+s_{m,k_m}]!}{(i_{m,k_m}+m+2s_{m,k_m}-2)!} \cdot \frac{C_{\ep}e^{-\frac{\ep}{t^2}}}{(\ep t)^{\frac{i_{m,k_m}+m}{2}+s_{m,k_m}}} \Big) \\
    &\le \frac{C_{\ep}^{I+L-J+1} e^{-\frac{\ep}{t^2}(L-J+1)}}{(\ep t)^{\M}} \cdot \Big( C^{\M} \cdot \frac{(L+1)^{\M}}{\M!} \Big).
\end{align*}

Since $T_{L,J;p} \le C^{L+J}$, we then have
\begin{align*}
    \sum_J \sum_I W^{I,J}_0 |x|^I &\le \sum_J C^J  \sum_{L\ge 2J} (\frac{C}{t})^{\theta\cdot L}  e^{-\frac{\ep}{t^2}(L-J+1)}\sum_{\substack{ L_0+L_1+L_2 = L \\ J_0 +J_1 +J_2 = J-1}} \sum_I \frac{(C|x|)^I \cdot (L+1)^{\M}}{(\ep t)^{\M}\cdot \M!}.
\end{align*}

If $|x| \le 1$, we simply use $(C|x|)^I \le C^I$ and then estimate
\begin{align*}
    \sum_I \frac{(C|x|)^I \cdot (L+1)^{\M}}{(\ep t)^{\M}\cdot \M!} \le \sum_I \frac{C^I \cdot (L+1)^{\M}}{(\ep t)^{\M}\cdot \M!} \le C^{L+J} e^{\frac{C(L+1)}{\ep t}}.
\end{align*}
This implies that
\begin{align*}
    \sum_J \sum_I W^{I,J}_0 |x|^I \le \sum_J C^J \sum_{L \ge 2J} (\frac{C}{t})^{\theta \cdot L} e^{-\frac{\ep}{t^2}(L-J+1)} \cdot e^{\frac{C(L+1)}{\ep t}}.
\end{align*}

Since $L \ge 2J$, we can let $K \equiv L  - 2J \ge 0$ and then 
\begin{align}\label{e:|x|<1}
    \sum_J \sum_I W^{I,J}_0 |x|^I &\le \sum_J C^{ J} \sum_{K \ge 0} (\frac{C}{t})^{\theta \cdot (K+2J)} e^{-\frac{\ep}{t^2}(J+K+1)} \cdot e^{\frac{C(K+2J+1)}{\ep t}} \\
    &\le e^{-\frac{\ep}{t^2}+\frac{C}{\ep t}}\sum_{J \ge 0} \Big( (\frac{C}{t})^{\theta } e^{-\frac{\ep}{t^2} +\frac{C}{\ep t}} \Big)^J \cdot \sum_{K \ge 0} \Big( (\frac{C}{t})^{\theta } e^{-\frac{\ep}{t^2} +\frac{C}{\ep t}} \Big)^K \notag\\
    &\le e^{-\frac{\ep}{t^2}+\frac{C}{\ep t}}(\sum_{J \ge 0} \Big(  e^{-\frac{\ep}{t^2} +\frac{C}{\ep t}} \Big)^J)^2,\notag
\end{align}
where the constant $C$ depends on $Q, \theta, \ep$.

If $|t| \le c(Q,\theta,\ep)$ for some small $c(Q,\theta)$, then $e^{-\frac{\ep}{t^2} +\frac{C}{\ep t}}  < \frac{1}{2}$ and then the series converges. Since $V^{I,J} \le W^{I,J}_0$, this implies that $u_1(x,t) = \sum_{I,J} V^{I,J}(t) x^I$ is well-defined in $[-1,1]\times [-c,c]$. 

Next we consider the case $|x| \ge 1$. If $I \le 2\M$, then $(C|x|)^{I} \le (C|x|)^{2\M}$ and thus
\begin{align*}
        \sum_I \frac{(C|x|)^I \cdot (L+1)^{\M}}{(\ep t)^{\M}\cdot \M!} \le C^{L+J} e^{\frac{C(L+1)x^2}{\ep t}}.
\end{align*}

If $I > 2\M$, then 
\begin{align*}
        \sum_I \frac{(C|x|)^I \cdot (L+1)^{\M}}{(\ep t)^{\M}\cdot \M!} \le C^{L+J} |x|^{I - 2\M} e^{\frac{C(L+1)x^2}{\ep t}} \le e^{\frac{C(L+J+1)x^2}{\ep t}}.
\end{align*}

Then by the same arguments we can obtain
\begin{align}
    \sum_J \sum_I W^{I,J}_0 |x|^I &\le  e^{-\frac{\ep}{t^2}+\frac{Cx^2}{\ep t}}(\sum_{J \ge 0} \Big(  e^{-\frac{\ep}{t^2} +\frac{Cx^2}{\ep t}} \Big)^J)^2,\label{e:|x|>1}
\end{align}
where the constant $C$ depends on $Q, \theta, \ep$.

Hence if $|x| \ge 1$, we can choose $|t| \le c/|x|^2$ for some small $c(Q,\theta)$ so that  $e^{-\frac{\ep}{t^2} +\frac{Cx^2}{\ep t}} < \frac{1}{2}$. Then the solution $u_1(x,t)$ is well-defined for $1\le |x|^2 \le c/|t|$. Combining with the discussion in the case when $|x|\leq 1$, we know that $u_1(x,t)$ is well-defined for $\{(x,t)\,|\, |x|\leq \frac{c}{|t|}\}$.

Moreover, $u_{1}(x,t)$ is a solution of \eqref{e:perturbequ} for $s=1$ with initial condition $u_1(x,0)\equiv 0$ and the solution is non-trivial since $u_1(0,t) = e^{-1/t^2}$.

\subsection{Most general region}

In this subsection, we generalize the result to the most general region using Lemma \ref{l:fun} and prove Theorem \ref{t:Tikhonov}. 
\begin{proof}[Proof of Theorem \ref{t:Tikhonov}]

Recall that in Subsection \ref{subsection:T}, we have defined a new family of Tikhonov type counter-examples
\begin{align*}
    \varphi_N(t)=\left\{\begin{array}{cc}
         f_N(t),&\, t\geq 0 \\
         0,&\, t<0.
    \end{array}\right.
\end{align*}
If we denote \begin{align*}
    g_N=\frac{2}{t}(\frac{8}{\ep})^{N-1}\prod_{j=0}^{N-1}\frac{1}{f_j^\ep},
\end{align*}
then we have 
\begin{align*}
    |\varphi_N^{(i)}|\leq i!(g_N)^if_N^{1-\ep}.
\end{align*}
Note that for $N=1$, this corresponds to the previous discussion. So if we build $u_{1,N}$ based on the new Tikhonov type counter-examples, from discussions in \eqref{e:|x|<1} and \eqref{e:|x|>1} we can get
\begin{align*}
    |u_{1,N}(x,t)|\leq
         f_N^{1-\ep}e^{Cg_N|\tilde{x}|^2} (\sum\limits_{J \ge 0} \Big( f_N^{1-\ep}e^{Cg_N|\tilde{x}|^2} \Big)^J)^2\leq \frac{ f_N^{1-\ep}e^{Cg_N|\tilde{x}|^2} }{(1- f_N^{1-\ep}e^{Cg_N|\tilde{x}|^2} )^2} \qquad\,\text{on}\,\{(x,t)\,|\,{Cg_N\tilde{x}^2}<\frac{1-\ep}{f_{N-1}}\,\},
\end{align*}
where $\tilde{x}=\max\{|x|,1\}$.

Also note that $\frac{1}{f_0}=\frac{1}{t^2}$ and \begin{align*}
    \frac{1}{f_{N+1}}=e^{\frac{1}{f_N}} 
\end{align*}
and thus \begin{align*}
    \frac{1}{f_N}\geq \frac{1}{t^2}
\end{align*}
while \begin{align*}
   (\frac{1}{f_N})(t)=   e^{\circ N}(\frac{1}{t^2}),
\end{align*}
where
\begin{align*}
    e^{\circ N}(x)=e^{e^{\circ {N-1}}(x)}.
\end{align*}

Thus by choosing $\ep$ small, $u_1(x,t)$ is well defined given 
\begin{align*}
    Cg_N\tilde{x}^4\leq -\frac{1}{3}\log (f_N) =\frac{1}{3f_{N-1}}.
\end{align*}
This is equivalent to that $u_1(x,t)$ is well defined given 
\begin{align*}
    \tilde{x}\leq (\frac{t\prod\limits_{j=0}^{N-1}f_j^\ep}{C_0C^N f_{N-1}})^{\frac{1}{4}}.
\end{align*}
Note that 
\begin{align*}
   (\frac{t\prod\limits_{j=0}^{N-1}f_j^\ep}{C_0C^N f_{N-1}})^{\frac{1}{4}}\geq C^{-\frac{N}{4}}t^{\frac{1}{4}}(\frac{\frac{1}{f_{N-1}^{1-\ep}}}{\prod\limits_{j=0}^{N-2}\frac{1}{f_j^\ep}} )^{\frac{1}{4}}.
\end{align*}
Also as for $y\geq 1$, \begin{align*}
 e^y\geq y^e,
\end{align*}

and by induction we have for $t\leq 1$ 
\begin{align*}
 ( \frac{1}{f_N})(t)= e^{\circ N}(\frac{1}{t^2})\geq ( e^{\circ N-1}({\frac{1}{t^2}}))^e\geq  e^{e^{N}\frac{1}{t^2}}.
\end{align*}
Thus given $t\leq 1$ and small $\ep$, for large  $k$ we would have 
\begin{align*}
     C^{-\frac{N}{4}}t^{\frac{1}{4}}(\frac{\frac{1}{f_{N-1}^{1-\ep}}}{\prod\limits_{j=0}^{N-2}\frac{1}{f_j^\ep}} )^{\frac{1}{4}}\geq& C^{-\frac{N}{4}}t^{\frac{1}{4}}(\frac{1}{f_{N-1}})^{\frac{1}{4}(1-\ep-\sum\limits_{j=1}^{N-2}\frac{1}{e^j})} \geq(e^{\circ N-1}(\frac{1}{t^2}))^{\frac{1}{4}(1-\frac{1}{e-1}-\ep-(1-\frac{1}{e-1}-\frac{2}{5}))}\\=&(e^{\circ N-1}(\frac{1}{t^2}))^{\frac{1}{4}(\frac{2}{5}-\ep)}\geq e^{\circ N-2}(\frac{1}{t^2}).
\end{align*}

So in particular we have that for all large $k$, we can get a corresponding $u_1$ defined on 
\begin{align*}
    \{(x,t)\,|\quad |x|\leq e^{\circ N}(\frac{1}{t^2}),\,\,-1\leq t< 0\}
\end{align*}

where
\begin{align*}
    e^{\circ N}(x)&=e^{e^{\circ {N-1}}(x)},\\ e^{\circ 0}(x)&=x.
\end{align*}

Also note that in this region, we have
 \begin{align*}
     f_N^{1-\ep}e^{{Cg_N\tilde{x}^2}}\leq f_N^{1-\ep}\frac{1}{f_N^{\frac{2}{3}}}=f_N^{\frac{1}{3}-\ep}.
 \end{align*}
 Thus for small $\ep$,
\begin{align*}
    |u_{s,N}(x,t)|\leq f_N^{\frac{1}{4}}\leq (e^{\circ N}(\frac{1}{t^2}))^{-\frac{1}{4}}\leq \frac{1}{e^{\circ N-1}(\frac{1}{t^2})}
\end{align*}
on \begin{align*}
    \{(x,t)\,|\quad |x|\leq e^{\circ N}(\frac{1}{t^2}),\,\,-1\leq t<0\}.
\end{align*}

As $e^{\circ N}(\frac{1}{t^2})$ is increasing with respect to $N$, by taking a larger $N$ , we finish the proof of Theorem \ref{t:Tikhonov}. 

\end{proof}

\section{Counter-examples for rescaled mean curvature flow}

In this section, we will prove Corollary \ref{c:RMCF} and Corollary \ref{c:prod}.

\begin{proof}[Proof of Corollary \ref{c:RMCF}]

For a rescaled mean curvature flow which is a graph of $v(x,\tau)$ over $\RR$, we have that 
\begin{align*}
    \partial_\tau v=L_{\RR}v+\mathcal{Q}(v,\na v,\na^2v).
\end{align*}
If we rescaled back to mean curvature flow, we have 
\begin{align*}
    u(y,t)={\sqrt{-t}}\cdot\,v(\frac{y}{\sqrt{-t}},-\ln(-t)).
\end{align*}
And thus $u$ satisfies
\begin{align*}
    \frac{\partial u}{\partial t}=\frac{\partial^2u}{\partial x^2}+\frac{1}{\sqrt{-t}}\mathcal{Q}(\frac{1}{\sqrt{-t}}u,\frac{\partial u}{\partial x},\sqrt{-t}\,\frac{\partial ^2u}{\partial x^2}).
\end{align*}
Note that from \cite{SX22}*{Proposition A.1}, we can see that $\mathcal{Q}$ is analytic around $(0,0,0)$ with respect to $(u,\frac{\partial u}{\partial x},\frac{\partial^2 u}{\partial x^2})$. So we can use Theorem \ref{t:Tikhonov} with $\beta=-\frac{1}{2},\,\,\alpha_1=-\frac{1}{2},\,\,\alpha_2=0,\,\,\alpha_3=\frac{1}{2}$ to get the Corollary \ref{c:RMCF}.

\end{proof}

\begin{proof}[Proof of Corollary \ref{c:prod}]

For a graphical rescaled mean curvature flow over $\RR \times \mathbb{S}^{1}$, where the graph function $v$ is independent with respect to $\mathbb{S}^1$, the equation is also of the form (see for example \cite{SX22})
\begin{align*}
    \partial_\tau v=L_{\RR} v+\mathcal{Q}(v,\na v,\na^2v),
\end{align*}
where $\mathcal{Q}$ is analytic around $(0,0,0)$ satisfying \eqref{e:Q}. So similar as before, if we rescaled back to mean curvature flow, we have 
\begin{align*}
    u(y,t)={\sqrt{-t}}\cdot\,v(\frac{y}{\sqrt{-t}},-\ln(-t)).
\end{align*}
And thus $u$ satisfies
\begin{align*}
    \frac{\partial u}{\partial t}= \frac{\partial^2u}{\partial x^2}+\frac{1}{\sqrt{-t}}\mathcal{Q}(\frac{1}{\sqrt{-t}}u,\frac{\partial u}{\partial x},\sqrt{-t}\,\frac{\partial^2u}{\partial x^2}).
\end{align*}As $\mathcal{Q}$ is analytic around $(0,0,0)$ with respect to $(u,\na u,\na^2 u)$ satisfies the assumptions in Theorem \ref{t:Tikhonov}. So from Theorem \ref{t:Tikhonov} with $\beta=-\frac{1}{2},\,\,\alpha_1=-\frac{1}{2},\,\,\alpha_2=0,\,\,\alpha_3=\frac{1}{2}$ we can get Corollary \ref{c:prod}.

\end{proof}

\appendix{}\section {An ODE lemma}

\label{ODE}

Here we prove an ODE lemma due to Filippas-Kohn \cite{FK92} and Merle-Zaag \cite{MZ98}, see also  Choi-Mantoulidis \cite{CM22}. .

\begin{Lem}\label{L:Merle-Zaag}
    Let $x_+, x_0, x_-: [0, \infty) \to [0, \infty)$ be absolutely continuous functions such that
    \begin{align*}
        x_+ + x_0 + x_- > 0,\\
        \liminf_{s \to \infty} x_+ = 0.
    \end{align*}
and 
\begin{align*}
    |x_0'| &\le \ep(x_+ + x_0 +x_-),\\
    x_+'  &\ge b_1x_+ -  \ep(x_0 + x_ -), \\
    x_-'  &\le - b_2x_- + \ep(x_0 + x_+),
\end{align*}
where $b_1\geq b>0,b_2\geq b>0$ are  positive constants.

Then there exist $\ep_0> 0$  such that if $\frac{\ep}{b} \le \ep_0$
\begin{equation}\label{e:x_+ is small}
    x_+ \le 2 \frac{\ep}{b} (x_0 + x_-) \text{ in } [0, \infty),
\end{equation}
and one of the following holds: 
\begin{align}
    \text{ either } x_- \le 8 \frac{\ep}{b} x_0 \text{ in } [s_0, \infty) \\
    \text{ or } x_0 \le 20\frac{\ep}{b} x_- \text {in } [0,\infty). \label{e: x_- dominates}
\end{align}

Moreover, if \eqref{e: x_- dominates} holds, then
\begin{equation*}
    \limsup_{s\to \infty} e^{(b_2-\ep)s} (x_+ + x_0 + x_-) = 0.
\end{equation*}
\end{Lem}

\begin{proof}
Firstly since once we define $\tilde{x}_\pm(t)=x_\pm(\frac{t}{b})$ and $\tilde{x}_0(t)=x_0(\frac{t}{b})$, we would have 
\begin{align*}
    |\tilde{x}_0'| &\le \frac{\ep}{b}(\tilde{x}_+ + \tilde{x}_0 +\tilde{x}_-),\\
    \tilde{x}_+'  &\ge \frac{b_1}{b}\tilde{x}_+ -  \frac{\ep}{b}(\tilde{x}_0 + \tilde{x}_ -), \\
    \tilde{x}_-'  &\le - \frac{b_2}{b}\tilde{x}_- +\frac{\ep}{b}(\tilde{x}_0 + \tilde{x}_+).
\end{align*}
So in the following, without loss of generality we only need to consider the case that $b=1$. Note that as $b_1,b_2\geq b$ and $x_+,x_-,x_0\geq 0$, we in particular have the following 
\begin{align*}
    |\tilde{x}_0'| &\le \frac{\ep}{b}(\tilde{x}_+ + \tilde{x}_0 +\tilde{x}_-),\\
    \tilde{x}_+'  &\ge \tilde{x}_+ -  \frac{\ep}{b}(\tilde{x}_0 + \tilde{x}_ -), \\
    \tilde{x}_-'  &\le - \tilde{x}_- +\frac{\ep}{b}(\tilde{x}_0 + \tilde{x}_+).
\end{align*}
    Let $\beta \equiv x_+ - 2\ep (x_0 + x_-)$. We first show that $\beta \le 0 $ in $[0, \infty)$.

    Suppose $\beta(s_0) > 0$, that is $x_+(s_0) \ge 2\ep (x_0 + x_-)(s_0)$. Then at $s_0$ we can compute
    \begin{align*}
        \beta' &= x'_+ - 2\ep (x'_0 + x'_-) \\
        &\ge \big(x_+ - \ep(x_0 + x_-)\big) - 2 \ep^2 \big(x_+ +x_0 +x_-) + 2 \ep \big( x_- - \ep(x_0 + x_+) \big) \\
        &\ge (1-4\ep^2) x_+ + (-\ep - 4\ep^2) x_0 + (\ep - 2\ep^2) x_- \\
        &\ge \ep(1 - 4\ep - 8 \ep^2) x_0 + \ep (3 -2\ep -8 \ep^2) x_-
    \end{align*}
where the last inequality comes from $\beta(s_0) \ge 0$. Hence if $\ep$ is small enough, then $\beta'(s_0) > 0$. In particular, we have $\beta(s) \ge \beta(s_0) > 0$ in $[s_0, \infty)$. On the other hand, since $\liminf_{s \to \infty} x_+(s) = 0$, we have $\liminf_{ s \to \infty} \beta(s) \le 0$. This is a contradiction. 

Hence we prove that $x_+ \le 2\ep (x_0 + x_-)$ in $[0, \infty)$. In particular, this implies that
\begin{align*}
    |x'_0| &\le 2 \ep (x_0 + x_-) \\
    x'_- &\le -\frac{1}{2} x_- + 2\ep x_0. 
\end{align*}

Now let $\gamma(s) \equiv 8 \ep x_0 - x_-$. 

\textbf{Case 1: } Suppose there exist some $s_0 \ge 0$ such that $\gamma(s_0)  = 0$. That is $8\ep x_0(s_0) = x_-(s_0)$. Then at $s_0$ we compute
\begin{align*}
    \gamma' &= 8 \ep x'_0 - x'_- \\
    &\ge -16 \ep^2 (x_0 + x_-) + \frac{1}{2} x_- - 2\ep x_0 \\
    &= 2\ep( 1- 8 \ep - 64 \ep^2) x_0\\
    &\ge 0.
\end{align*}
From the inequality, we note that if $\gamma'(s_0) = 0$, then $x_0(s_0) = x_-(s_0) = 0$. This implies that $x_+(s_0) = 0$. Contradiction arises. Hence we have $\gamma'(s_0) > 0$. In particular $\gamma(s) \ge \gamma(s_0) \ge 0$ for $s \in [s_0, \infty)$.

\textbf{Case 2: } Suppose $\gamma(s) < 0$ for $s\in [0, \infty)$. This implies that
\begin{align*}
    |x'_0| &\le (\frac{1}{4} + 2 \ep) x_-, \\
    x'_- &\le -\frac{1}{4} x_-. 
\end{align*}

Since $x_-(t) \ge 0$, then we have
\begin{equation*}
    x_-(s) \ge \frac{1}{4} \int_{s}^{\infty} x_-(t) dt,
\end{equation*}
and for any large $T$ we have
\begin{equation*}
    x_0(s) \le x_0(T) + (\frac{1}{4} + 2 \ep) \int_s^T x_-(t) dt \le x_0(T) + (1 + 8 \ep) x_-(s).
\end{equation*}

Since $\int_s^{\infty} x_- \le 4 x_-(s) < \infty$, there exists a sequence $T_i \to \infty$ such that $x_-(T_i) \to 0$. Since $\gamma(T_i) < 0$, then $x_0(T_i) \le (8\ep)^{-1} x_-(T_i) \to 0$ as $T_i \to \infty$. Hence we have
\begin{equation*}
    x_0(s) \le (1 + 8 \ep) x_-(s).
\end{equation*}

Hence we obtain an improved estimate for $x_0$. Using this we have
\begin{equation*}
    |x'_0|\le 2\ep (1+8\ep +1 ) x_-,
\end{equation*}
and thus
\begin{equation*}
    x_0(s) \le 2\ep (2+8\ep) \int_s^{\infty} x_-(t) dt \le  8\ep (2+8\ep) x_-(s). 
\end{equation*}

Hence, if $\ep$ is chosen small, then we have $x_0(s) \le 20\ep x_-(s)$ in $[0,\infty)$. In this case, we have
\begin{equation*}
    x'_- \le -b_2x_- + \ep (x_0 + x_+) \le (-b_2 + \ep/2)x_-.
\end{equation*}

By Fundamental Theorem of Calculus,
\begin{equation*}
    \limsup_{s \to \infty} e^{(b_2-\ep/2)s} x_-(s) \le
    x_-(0).
\end{equation*}

This implies that
\begin{equation*}
    \limsup_{s \to \infty} e^{(b_2-\ep)s} x_-(s) \le 0.
\end{equation*}

Moreover, by \eqref{e:x_+ is small} and \eqref{e: x_- dominates}, we obtain 
\begin{equation*}
    \limsup_{s\to \infty} e^{(b_2- \ep)s} (x_+ + x_0 + x_-) = 0.
\end{equation*}
\end{proof}

\section{Combinatorial lemmas}\label{A:Catalan}

We collect some useful facts in combinatorics here:

\begin{Lem}\label{l:sumprod}
Given $X, Y,Z$ and $\sum_i l_i = X+Y+Z$:
\begin{equation}\label{e:counting row and column}
    \sum_{\substack{\sum_i x_i = X\\
\sum_i y_i = Y\\
\sum_i z_i = Z\\ x_i+y_i+z_i= l_i}} \prod_i \frac{l_i!}{x_i!y_i!z_i!} = \frac{(X+Y+Z)!}{X!Y!Z!}.
\end{equation}    
\end{Lem}

\begin{Lem}[Robbins' inequality]\label{l:Robbin}
    \begin{align*}
        \sqrt{2\pi n} \cdot \Big( \frac{n}{e} \Big)^n \cdot e^{\frac{1}{12n+1}} < n! < \sqrt{2\pi n} \cdot \Big( \frac{n}{e} \Big)^n \cdot e^{\frac{1}{12n}}.
    \end{align*}    
\end{Lem}

\subsection{Catalan number}\label{sec:cat}
Let us recall some properties of Catalan number. Set $C_n=\frac{1}{n}\binom{2n-2} {n-1}$. Since we have
\begin{align*}
    \sum_{n=1}^\infty C_nx^n=\frac{1-\sqrt{1-4x}}{2},
\end{align*}
we would get
\begin{align*}
    (\sum_{n=1}^\infty C_nx^n)^2-\sum_{n=1}^\infty C_nx^n+x=0.
\end{align*}

This corresponds to
\begin{align*}
    \sum_{n=1}^\infty C_nx^n=\frac{x}{1- \sum_{n=1}^\infty C_nx^n}=x(1+\sum_{m=1}^\infty(\sum_{k=1}^\infty C_kx^k)^m).
\end{align*}
Note that we also have 
\begin{align*}
    \sum_{n=1}^\infty C_nx^n=(\sum_{n=1}^\infty C_nx^n)^2+x.
\end{align*}
So we have that 
\begin{Lem}\label{l:Catalan variable branching}
    \begin{align}
    C_n&=\sum_{m=1}^{n-1}\sum\limits_{\sum\limits_{i=1}^mk_i=n-1}\prod\limits_{i=1}^mC_{k_i},\\ 
    C_1&=1.
\end{align}

Also  we have

\begin{equation*}
    C_n = \sum_{m=1}^{n-1} C_{m} C_{n-1-m}, C_1 =1.
\end{equation*}
\end{Lem}

\subsection{A product lemma}

\begin{Lem}\label{l:prod}
    Given positive numbers $i=\sum_{j=1}^ki_j$, we have that for any $v=0,1,2$, 
    \begin{align}
        \label{e:prod}\prod_{j=1}^k\frac{i_j!}{(2i_j-v)!}\leq (\frac{[\frac{i}{k}]!}{(2[\frac{i}{k}]-v)!})^{k-i+k[\frac{i}{k}]}(\frac{([\frac{i}{k}]+1)!}{(2[\frac{i}{k}]+2-v)!})^{i-k[\frac{i}{k}]}.
    \end{align}
\end{Lem}
\begin{proof}
    Since given $i$, there is only finitely many choice of $i_j$, without loss of generality let us assume that $i_j$ attains the maximum among all possible choices.
    
    We claim that given the choice attains the maximum, for any $i_l,i_m$ we have, \begin{align*}
        |i_l-i_m|\leq 1.
    \end{align*}
    Let us prove the claim by contradiction. If $|i_1-i_2|\geq 2$, without loss of generality we could assume that $i_2\geq i_1+2\geq 3$. Given that $i_j$ attains the maximum, in particular we have
    \begin{align*}
        \frac{(i_1+1)!}{(2i_1+2-v)!}\frac{(i_2-1)!}{(2i_2-2-v)!}\leq \frac{i_1!}{(2i_1-v)!} \frac{i_2!}{(2i_2-v)!}.
    \end{align*}
Thus 
\begin{align*}
    \frac{i_1+1}{(2i_2+2-v)(2i_2+1-v)}\leq\frac{i_2}{(2i_2-v)(2i_2-1-v)}.
\end{align*}

Since\begin{align*}
    \frac{d}{dx}(\frac{x}{(2x-v)(2x-1-v)})=&\frac{x}{(2x-v)(2x-1-v)}(\frac{1}{x}-\frac{2}{2x-v}-\frac{2}{2x-1-v})\\=&\frac{1}{(2x-v)^2(2x-1-v)^2}((2x-v)(2x-1-v)-2x(2x-v-1)-2x(2x-v))\\ =&\frac{1}{(2x-v)^2(2x-1-v)^2}(v+v^2-4x^2),\\
\end{align*}
and $v+v^2-4x^2\leq 6-4x^2<0$ for $x\geq 2$, combining with that $i_1+1,i_2\geq2$ we would have that $i_1+1\geq i_2$. This contradicts with that $i_2\geq i_1+2$. So we proved the claim that  
    \begin{align*}
        |i_l-i_m|\leq 1,\,\,\forall\,l,m\in\{1,2,\cdots,k\}.
    \end{align*}

Combining with that $\sum_{j=1}^ki_j=i$, we obtained the \eqref{e:prod}.
    
\end{proof}

Combining this with Robbin inequality \ref{l:Robbin}, we obtain the following estimate:
\begin{Lem}\label{l:prod estimate}
    Given positive integers $i=\sum_{j=1}^k i_j$, then for any $v=0,1,2$
    \begin{equation*}
        \prod_{j=1}^k \frac{i_j!}{(2i_j - v)!} \le C^{i} \cdot (\frac{k}{i})^{i-5k} \le C^i \cdot \frac{k^{i}}{i!}.
     \end{equation*}
\end{Lem}

\begin{proof}
It is easy to see that the result follows when $i \le 100k$. In the following we only consider the case where $i > 100k$.
    By \eqref{e:prod}, we have
    \begin{align*}
        \prod_{j=1}^k \frac{i_j!}{(2i_j - v)!} \le \Big(\frac{([\frac{i}{k}]+1)!}{(2 [\frac{i}{k}] - 2)!} \Big)^k.
    \end{align*}

Then by Robbin's inequality \ref{l:Robbin}, we have
\begin{align*}
    \Big(([\frac{i}{k}]+1)! \Big)^k \le C \cdot \frac{\sqrt{[\frac{i}{k}]+1}}{e^{[\frac{i}{k}] +1 }} \cdot ([\frac{i}{k}]+1)^{([\frac{i}{k}]+1) \cdot k} \le C \cdot \frac{\sqrt{[\frac{i}{k}]+1}}{e^{[\frac{i}{k}] +1 }} \cdot \Big( \frac{i+k}{k} \Big)^{i+k},
\end{align*}
and
\begin{align*}
    \Big((2[\frac{i}{k}] -2)! \Big)^k &\ge C \cdot \frac{\sqrt{2[\frac{i}{k}]-2}}{e^{2[\frac{i}{k}] -2 }} \cdot (2[\frac{i}{k}] -2)^{(2[\frac{i}{k}] -2) \cdot k} \\
    &\ge C \cdot \frac{\sqrt{2[\frac{i}{k}]-2}}{e^{2[\frac{i}{k}] -2 }} \cdot (\frac{2i}{k} - 4)^{(\frac{2i}{k} -4)\cdot k} \\
    &\ge C \cdot \frac{\sqrt{2[\frac{i}{k}]-2}}{e^{2[\frac{i}{k}] -2 }} \cdot \Big( \frac{2i - 4k}{k} \Big)^{2i - 4k}.
\end{align*}

Since $i>100k$, then we have $i+k < 1.1 i$ and $2i-4k > 1.9 i$. 
Therefore,
\begin{align*}
    \Big(\frac{([\frac{i}{k}]+1)!}{(2 [\frac{i}{k}] - 2)!} \Big)^k &\le C^{i} \cdot \frac{(i+k)^{i+k}}{(2i-4k)^{2i-4k}} \cdot \frac{k^{2i-4k}}{k^{i+k}} \\
    &\le  C^i \cdot \frac{(1.1 i)^{i+k}}{(1.9i) ^{2i-4k}} \cdot k^{i-5k} \\
    &\le C^i \cdot \frac{i^{k}}{i^{i-4k}} \cdot k^{i-5k}\\
    &\le C^i \cdot (\frac{k}{i})^{i-5k} \\
    &\le C^i \cdot \frac{k^{i}}{i!}.
\end{align*}

\end{proof}

\section{Lemmas on functions}\begin{Lem}   
For functions $g\geq 0$ and $f<0$,  if $|f^{(i)}|\leq i!g^i|f|$, we would have that
\begin{align*}|(e^f)^{(i)}|<i!e^{(1-\ep)f}(\frac{2}{\ep}g)^i.\end{align*}\label{l:func1}
\end{Lem}

\begin{proof}    
Also we have that     
\begin{align*}    (e^f)^{(i)}=e^f\sum_{k=1}^{i }\sum_{\sum\limits_{j=1}^k i_j=i} \frac{i!}{k!\prod\limits_{j=1}^ki_j!}\prod_{j=1}^k f^{(i_j)}.  
\end{align*}

Combining with assumptions we have 
\begin{align*}    | (e^f)^{(i)}|\leq e^fg^i\sum_{k=1}^i\frac{i!}{k!}|f|^{k}{ \binom{i-1}{k-1}}\leq i!e^f(\frac{2}{\ep}g)^i\sum_{k=1}^{i}\frac{(\ep|f|)^{k}}{k!}\leq  i!e^f(\frac{2}{\ep}g)^ie^{-\ep f}= i!e^{(1-\ep)f}(\frac{2}{\ep}g)^i
\end{align*}

\end{proof}

\begin{Lem} \label{l:func2}
  For functions $g\geq 0$ and $f$,  if $|f^{(i)}|\leq i!g^i|f|^{1-\ep}$, we would have that
    \begin{align*}
        (\frac{1}{f})^{(i)}\leq i!(4g)^i(\frac{1}{|f|})^{1+i\ep}.
    \end{align*}
\end{Lem}

\begin{proof}
     We have that     
\begin{align*}    (\frac{1}{f})^{(i)}=\sum_{k=1}^{i }\frac{(-1)^kk!}{f^{k+1}}\sum_{\sum\limits_{j=1}^k i_j=i} \frac{i!}{k!\prod\limits_{j=1}^ki_j!}\prod_{j=1}^k f^{(i_j)}=\sum_{k=1}^{i }\frac{(-1)^k}{f^{k+1}}\sum_{\sum\limits_{j=1}^k i_j=i} \frac{i!}{\prod\limits_{j=1}^ki_j!}\prod_{j=1}^k f^{(i_j)}.  
\end{align*}

Combining with assumptions we can get
\begin{align*}
    |(\frac{1}{f})^{(i)}|\leq g^i\sum_{k=1}^ii!|f|^{-1-k\ep}\binom{i-1} {k-1}\leq i!(4g)^i(\frac{1}{|f|})^{1+i\ep}    ,
\end{align*}

\end{proof}

\begin{Lem}\label{l:fun}
    Let us take $f_0=t^2$ and $f_{k+1}=e^{-\frac{1}{f_k}}$, $k\geq 0$. Then we have that 
    \begin{align*}
        |f_k^{(i)}|\leq i!f_k^{1-\ep}(\frac{2}{t}(\frac{8}{\ep})^{k-1}\prod_{j=0}^{k-1}\frac{1}{f_j^\ep})^i  .                           
    \end{align*}
\end{Lem}
\begin{proof}
From Lemma \ref{l:func1}  we have that if $|f_{k}^{(i)}|\leq i!g_k^if_k^{1-\ep}$, we can get
\begin{align*}
        |(-\frac{1}{f_k})^{(i)}|\leq i!(\frac{4g}{f_k^\ep})^i(\frac{1}{f_k})                .
\end{align*}
Combining with Lemma \ref{l:func2} we have
\begin{align*}
    |f_{k+1}^{(i)}|\leq i!f_{k+1}^{1-\ep}(\frac{8g_k}{\ep f_k^\ep})^i.
\end{align*}
 Thus we can get the recursive formula 
 \begin{align*}
     g_{k+1}=\frac{8}{\ep f_k^\ep}g_k.
 \end{align*}
Combining with $g_0=\frac{2}{t}$, the proof is finished.

\end{proof}

\end{document}